\documentclass[12pt, a4paper, reqno]{amsart}
\usepackage[margin=3cm]{geometry}

\usepackage{latexsym,amsmath,amsthm,amssymb,mathrsfs}
\usepackage{epsf}
\usepackage[latin1]{inputenc}
\usepackage{graphics}
\usepackage{enumerate}
\usepackage{calc}
\usepackage[active]{srcltx}
\usepackage{epsfig}
\usepackage{pdfsync}

\usepackage{color}

\usepackage{amscd}
\usepackage{esint}

\usepackage{ulem}

\usepackage{bbm}

\newcommand{\stkout}[1]{\ifmmode\text{\sout{\ensuremath{#1}}}\else\sout{#1}\fi}

\newcommand{\F}{\mathcal{F}}
\renewcommand{\S}{\mathbb{S}}

\newcommand{\la}{\lambda}

\newcommand{\D}{\nabla}
\newcommand{\ra}{r_{\!A}}

\def\d{\partial}

\newcommand{\h}[1]{\widehat{#1}}
\def\e{\varepsilon}

\def\db{\bar{d}}

\def\sgn{\text{ sgn }}

\def\weak{\rightharpoonup}
\newcommand{\R}{\mathbb{R}}

\def\NN{\mathbb{N}}

\def\FF{\mathcal{F}}

\def\AA{\mathcal{A}}

\def\D{\nabla}

\newcommand{\dis}{\displaystyle}

\newcommand{\La}{\Lambda}
\renewcommand{\P}{\mathcal{P}}

\def\XXint#1#2#3{{\setbox0=\hbox{$#1{#2#3}{\int}$}
     \vcenter{\hbox{$#2#3$}}\kern-.5\wd0}}

\newcounter{bei}

\newcommand{\A}{\mathfrak{A}}

\renewcommand{\t}{\widetilde}
\renewcommand{\o}{\overline}
\renewcommand{\u}{\underline}

\newcommand{\De}{\nabla_{\!\e}}

\numberwithin{equation}{section}

\theoremstyle{plain}
\begingroup
\newtheorem{theorem}{Theorem}[section]
\newtheorem{lemma}[theorem]{Lemma}
\newtheorem{proposition}[theorem]{Proposition}

\endgroup

\theoremstyle{definition}
\begingroup
\newtheorem{definition}[theorem]{Definition}
\newtheorem{remark}[theorem]{Remark}
\newtheorem{example}[theorem]{Example}
\endgroup

\DeclareMathOperator{\sym}{sym}

\author[L. Freddi]{Lorenzo Freddi}
\author[P. Hornung]{Peter Hornung}
\author[M.G. Mora]{Maria Giovanna Mora}
\author[R. Paroni]{Roberto Paroni}

\address[L. Freddi]{Dipartimento di Scienze Matematiche, Informatiche e Fisiche,
via delle Scienze 206, 33100 Udine, Italy
}\email{lorenzo.freddi@uniud.it}
\address[P. Hornung]{Fachrichtung Mathematik,
TU Dresden,
01062 Dresden,
Germany
}\email{peter.hornung@tu-dresden.de}
\address[M.G. Mora]{Dipartimento di Matematica, Universit\`a di Pavia, via Ferrata 1, 27100 Pavia, Italy}
\email{mariagiovanna.mora@unipv.it}
\address[R. Paroni]{Dipartimento di Ingegneria Civile e Industriale, Universit\`{a} di
Pisa, Largo Lucio Lazzarino 1,
56122 Pisa, Italy} \email{roberto.paroni@unipi.it}

\begin{document}
\title[The Sadowsky functional with boundary conditions]
{Stability of boundary conditions for the Sadowsky functional}

\date{}

\begin{abstract}
It has been proved by the authors that the (extended) Sadowsky functional can be deduced as the $\Gamma$-limit of
the Kirchhoff energy on a rectangular strip of height $\e$, as $\e$ tends to $0$.
In this paper we show that this $\Gamma$-convergence result is stable when affine boundary conditions are prescribed
on the short sides of the strip. These boundary conditions include those corresponding to a M\"obius band.
\end{abstract}

\maketitle

\section{Introduction}\label{Sec0}

The derivation of variational models for thin structures is one of the most fruitful applications of $\Gamma$-convergence  in continuum mechanics.  
A typical example is a variational model for a two-dimensional structure obtained as a $\Gamma$-limit of energies for bodies occupying cylindrical regions whose heights tend to zero.
Quite often, such variational derivations focus on the asymptotic behavior of the bulk energy while partially or completely neglecting the contribution of external forces and/or boundary conditions. Usually, external forces, such as, e.g., dead loads, can be easily included in the analysis afterwards by using the stability of $\Gamma$-convergence with respect to continuous additive perturbations. On the other hand, boundary conditions may affect the $\Gamma$-limit of the bulk energy and, even if not so, they must be taken into account in the construction of the so-called recovery sequence. 
In some cases boundary conditions may be an essential feature of the structure under study: consider, for instance, a M\"obius band, where the two ends of the strip are glued together after a half-twist. 

In this paper we show that the $\Gamma$-convergence result leading to the derivation of the Sadowsky functional  (see \cite{FrHoMoPa,longversion})
is stable with respect to an appropriate set of boundary conditions, that include those corresponding to a M\"obius band.

In the last years the Sadowsky functional and, more in general, the theory of elastic ribbons have received a great deal of attention. 
Part of the reappraisal on the subject is due to the work \cite{SH2007} by
Starostin and van der Heijden on elastic M\"obius strips. Since then the literature has been increasing in several directions, as partially documented in the book \cite{FoFr}
edited by Fosdick and Fried. 
Indeed, the mechanics of M\"obius elastic ribbons has been studied, e.g., in \cite{BH2015,MH18,SH2015}.
The morphological stability of ribbons has been considered in \cite{AS,CDD,LSSM,MH18} and their helicoidal-to-spiral transition in \cite{ADK,PT,TeV,ToV}.
The relation between rods and ribbons, as well as the derivation of viscoelastic models, has been investigated in
\cite{AN,BFV,DA2015,FM}, while models of ribbons with moderate displacements have been deduced in \cite{Da,FrHoMoPa2,FrMoPa2}.
Finally, for numerics and experiments on ribbons we refer to \cite{Ba,CDNR,KAB,KuHB,Yu,YDMG,YH}.

The Sadowsky functional has been introduced by Sadowsky in 1930 as a formal limit of the Kirchhoff energy for a M\"obius band of vanishing width (see \cite{HF2015b,HF2015a,Sadowsky2}).
In \cite{FrHoMoPa,longversion} this derivation has been made precise, using the language of $\Gamma$-convergence, 
for a narrow ribbon without any kind of boundary conditions or topological constraints.
More precisely, let $S_\e=(0,\ell)\times (-\e/2,\e/2)$ be the reference configuration of an inextensible isotropic strip,
where the width $\e$ is much smaller that the length $\ell$. The Kirchhoff energy of the strip is 
$$
E_\e(u)= \frac{1}{\e}\int_{S_{\e}} |\Pi_u(x)|^2\, dx,
$$
where $u:S_\e\to\R^3$ is a $W^{2,2}$-isometry and $\Pi_u$ is the second fundamental form of the surface $u(S_\e)$. 
Note that by Gauss's Theorema Egregium $\det \Pi_u=0$.
In \cite{FrHoMoPa,longversion} it has been proved that the $\Gamma$-limit of $E_\e$, as $\e\to0$, 
provides an extension of the classical Sadowsky functional 
and is given by
\begin{equation}\label{defJ}
E(y,d_1,d_2,d_3)=\int_{0}^{\ell} \o Q(d_1'\cdot d_3, d_2'\cdot d_3)\,ds
\end{equation}
where the unit vectors $d_i$ are such that $r^T:=(d_1|d_2|d_3)\in W^{1, 2}((0,\ell); SO(3))$ 
and satisfy the nonholonomic constraint 
\begin{equation}\label{BC}
d_1^\prime\cdot d_2=0,
\end{equation}
while the deformation $y$ of the centerline of the strip is related to the system of directors by the equation
\begin{equation}\label{BC1}
y'=d_1.
\end{equation}
In other words, the director $d_1$ is the tangent vector to the deformed centerline. The director $d_2$ describes the ``transversal orientation'' of the deformed strip, hence
the constraint \eqref{BC} means that the strip cannot bend within its own plane. The director $d_3$ represents the normal vector to the deformed strip.
The energy depends on the two quantities $d_1'\cdot d_3$ and $d_2'\cdot d_3$, that represent the bending strain and the twisting strain of the strip, respectively. 
The limiting energy density $\overline Q$ is given by
\begin{equation}\label{barQ}
\overline Q(\mu,\tau) = \begin{cases}
\dfrac{(\mu^2+\tau^2)^2}{\mu^2} & \text{ if } |\mu|>|\tau|,
\smallskip\\
4\tau^2 & \text{ if } |\mu|\leq|\tau|,
\end{cases}
\end{equation}
hence it is a convex function that coincides with the classical Sadowsky energy density for $|\mu|>|\tau|$.

The $\Gamma$-convergence result in \cite{FrHoMoPa,longversion} is supplemented by suitable compactness properties, that
guarantee convergence of minimizers of the Kirchhoff energy to minimizers of the (extended) Sadowsky functional \eqref{defJ}.

As already mentioned, these results were proved without any kind of boundary conditions or topological constraints.
Therefore, the question remained open of whether the Sadowsky functional correctly describes the behavior of narrow M\"obius bands or of other closed narrow ribbons.
 
In this paper we answer this question by considering prescribed affine boundary conditions on the short sides $\{0\}\times (-\e/2,\e/2)$ and $\{\ell\}\times (-\e/2,\e/2)$.
We prove (see Theorem~\ref{Gamma}) that, as $\e\to0$, the $\Gamma$-limit of the Kirchhoff energy is still given by the Sadowsky functional
\eqref{defJ}, where now the frame $(y,r)$ satisfies, in addition to \eqref{BC}--\eqref{BC1}, a set of boundary conditions that we now describe.

By translating and rotating the coordinate system, we
may assume, with no loss of generality, that the side $\{0\}\times (-\e/2,\e/2)$ 
is clamped and that, on this short side, the ribbon is tangent to the undeformed centerline of the ribbon.
In the limit problem this boundary condition leads to
\begin{equation}\label{bc-intro}
y(0) =0\quad \mbox{ and } \quad  r(0) = I,
\end{equation}
where $I$ is the identity matrix.
Similarly, the boundary condition on the side 
$\{\ell\}\times (-\e/2,\e/2)$
leads to
\begin{equation}\label{bc-intro2}
y(\ell) = \o y \quad \mbox{ and } \quad r(\ell) = \o r,\
\end{equation}
where $\o y\in \R^3$ is the ending point of the deformed centerline 
and $\o r\in SO(3)$ is the orientation in the deformed configuration of the short side at $\ell$.

We note that these boundary conditions can model both open and closed ribbons. In particular, 
a M\"obius band satisfies the above boundary conditions with $\o y=0$ and $\o r=(e_1|-e_2|-e_3)$.
However, we point out that the geometrical boundary conditions \eqref{bc-intro}--\eqref{bc-intro2} are insensitive to the number of full turns of the director $d_2$ along the centerline, thus a closed ribbon with an odd number of half-twists satisfies the same boundary conditions as a M\"obius band.
Alternatively, one may prescribe both the boundary conditions and the linking number of the strip (see, e.g., \cite{AlexanderAntman}). This will be addressed in future work.

We close this introduction by discussing the main mathematical difficulties in the proof of our main result (Theorem~\ref{Gamma}).
Compactness and the liminf inequality can be proved by relying on the results of \cite{FrHoMoPa,longversion}. On the other hand, 
the construction of the recovery sequence has to be modified in a non trivial way to satisfy the prescribed boundary conditions on the short sides of the strip.
The strategy in \cite{FrHoMoPa,longversion} is as follows. The limiting energy density $\overline Q$ is obtained by relaxing the zero determinant constraint with respect to the weak convergence in $L^2$. More precisely, \cite[Lemma~3.1]{FrHoMoPa} shows that the lower semicontinuous envelope of the functional
$$
M\in L^2((0,\ell);\R^{2\times 2}_{\sym}) \ \mapsto \ \begin{cases}
\displaystyle\int_0^\ell |M(s)|^2\, ds & \text{ if } \det M=0,
\\
+\infty & \text{ otherwise,}
\end{cases}
$$
with respect to the weak topology of $L^2((0,\ell);\R^{2\times 2}_{\sym})$ is given by
$$
M\in L^2((0,\ell);\R^{2\times 2}_{\sym}) \ \mapsto \int_0^\ell \big( |M(s)|^2+2|\det M(s)|\big)\, ds.
$$
The energy density $\overline Q$ is then defined by the minimization problem
\begin{equation}\label{barQ-min}
\overline Q(\mu,\tau) = \min_{\gamma\in\R}\Big\{ |M|^2+2|\det M| : \ M=\Big(\begin{array}{cc} \mu & \tau \\ \tau & \gamma \end{array}\Big)\Big\},
\end{equation}
from which equation \eqref{barQ} follows. Given a frame $(y,r)$ satisfying \eqref{BC}--\eqref{BC1}, one can define $\mu:=d_1'\cdot d_3$,
$\tau:=d_2'\cdot d_3$, and $\gamma$ as a solution of the minimization problem \eqref{barQ-min}. By the relaxation result there exists a sequence $(M^j)\subset L^2((0,\ell);\R^{2\times 2}_{\sym})$ with $\det M^j=0$ for every $j$, such that $(M^j)_{11}\weak\mu$ and $(M^j)_{12}\weak\tau$ weakly in $L^2(0,\ell)$, and
\begin{equation}\label{int-relax}
\int_0^\ell |M^j(s)|^2\, ds \ \to \ \int_0^\ell \overline Q(\mu,\tau)\, ds = E(y,d_1,d_2,d_3).
\end{equation}
By approximation one can assume $M^j$ to be smooth and by a diagonal argument it is enough to construct a recovery sequence for $M^j$ for every $j$.
To do so, we first build a new frame $(y^j,r_j)$ satisfying \eqref{BC}--\eqref{BC1} and such that $(d^j_1)^\prime\cdot d^j_3=(M^j)_{11}$ and
$(d^j_2)^\prime\cdot d^j_3=(M^j)_{12}$. Finally, for $\e$ small enough we construct smooth isometries $u_\e:S_\e\to\R^3$ such that 
$\nabla u_\e(\cdot,0)=(d^j_1|d^j_2)$ and $\Pi_{u_\e}(\cdot,0)=M^j$ on $(0,\ell)$. These two properties guarantee that $(u_\e)$ is a recovery sequence.

In the presence of boundary conditions, the given frame $(y,r)$ satisfies in addition \eqref{bc-intro}--\eqref{bc-intro2}. In this case one has first to ensure that
the auxiliary frames $(y^j,r_j)$ still satisfy \eqref{bc-intro}--\eqref{bc-intro2} and then construct the smooth isometries $u_\e$ in such a way that
the boundary conditions are met on the short sides of the strip. The key ingredient to do so is a refinement of the relaxation result \cite[Lemma~3.1]{FrHoMoPa}, showing that
the sequence $(M^j)$ in \eqref{int-relax} can be modified in such a way to accomodate the boundary conditions (see Proposition~\ref{haupt}).
This is based on density results for framed curves preserving boundary conditions, proved in \cite{framed}.

In the last part of the paper we derive the Euler-Lagrange equations for the functional \eqref{defJ} with boundary conditions
\eqref{bc-intro}--\eqref{bc-intro2}
and we show, under some regularity assumptions, that the centerline of a developable M\"obius band at equilibrium cannot be a planar curve.\bigskip

\noindent
{\bf Plan of the paper.} In Section~\ref{GUBC} we set the problem, we prove compactness for deformations with equibounded energy, and we state the $\Gamma$-convergence result.
Section~\ref{UB} contains the approximation results that are key to prove the existence of a recovery sequence in Section~\ref{RS}.
In Section~\ref{eqeq} we derive the equilibrium equations for the limit boundary value problem and
in the last section we focus on regular solutions of this problem in the case of a M\"obius band.\bigskip

\noindent
{\bf Notation.} Along the whole paper $(e_1,e_2,e_3)$ and $(\u e_1,\u e_2)$ denote the canonical bases of $\R^3$ and $\R^2$, respectively.  


\section{Setting of the problem and main result}\label{GUBC}

In this section we recall the setting of the problem and state the main result (Theorem~\ref{Gamma}), which proves stability of $\Gamma$-convergence with respect to an appropriate set of geometric boundary conditions.
 
We consider the interval $I=(0,\ell)$ with $\ell>0$.
For $0<\e\ll 1$ let $S_{\e} = I\times (-\e/2, \e/2)$ be the reference configuration of an inextensible elastic narrow strip.
We assume the energy density of the strip, $Q:\R^{2\times 2}_{\sym} \to[0,+\infty)$,  to be an isotropic and quadratic function of the second fundamental form. The Kirchhoff energy of the strip is 
$$
E_{\e}(u) = \frac{1}{\e}\int_{S_{\e}} Q(\Pi_u(x))\, dx,
$$
where the second fundamental form of $u$, 
$\Pi_u : S_\e\to\R^{2\times 2}_{\sym}$, is defined by
$$
(\Pi_u)_{ij} = n_u\cdot\d_i\d_j u,
$$
and
$$
n_u = \d_1 u\times\d_2 u
$$
is the unit normal to $u(S_\e)$. 
Due to the inextensibility constraint deformations $u : S_\e\to \R^3$ satisfy the relations
 $\partial_i u\cdot\partial_j u = \delta_{ij}$, where $\delta_{ij}$ is the Kronecker delta.
We denote the space of $W^{2,2}$-isometries of $S_\e$ by 
$$
W^{2,2}_{\rm iso}(S_\e;\R^3):= \big\{ u\in W^{2,2}(S_\e;\R^3): \ \partial_i u\cdot\partial_j u=\delta_{ij} \text{ a.e.\ in } S_\e\big\}.
$$

Since the energy density $Q$ is isotropic, it depends on $\Pi_u$ only through the trace and the determinant of  $\Pi_u$.
On the other hand, the inextensibility constraint and Gauss's Theorema Egregium imply that the Gaussian curvature is equal to zero, that is,
$\det \Pi_u=0$. Thus, the energy may be expressed in terms of $\mbox{tr}\,\Pi_u$ only. Equivalently, since $|\Pi|^2+ 2\, \mbox{det}\, \Pi=(\mbox{tr}\,\Pi)^2$
for every $\Pi\in \R^{2\times 2}_{\sym}$, we can write the energy in terms of the norm of  $\Pi_u$. With these considerations in mind, up to a multiplicative constant, 
the energy can be written as
$$
E_{\e}(u) = \frac{1}{\e}\int_{S_{\e}} |\Pi_u(x)|^2\, dx. 
$$

We shall require deformations $u$ to satisfy ``clamped'' boundary conditions at
$x_1=0$ and $x_1=\ell$. By composing deformations with a rigid motion, we may assume, without loss of generality, that
$$
u(0,x_2) = x_2,\quad \d_1 u(0, x_2) = e_1\qquad \mbox{ for }x_2\in(-\e/2,\e/2).
$$
To set the boundary conditions at $x_1=\ell$, we fix $\o y\in\R^3$ and $\o r^T=(\db_1|\db_2|\db_3)\in SO(3)$
and require that
$$
u(\ell,0)=\o y,\quad \d_i u(\ell, x_2) = \o d_i\qquad \mbox{ for }x_2\in(-\e/2,\e/2),
$$
for $ i=1,2$. We note that the imposed boundary conditions keep straight the sections at $x_1=0$ and $x_1=\ell$,
or, in other words, the image of $u(0, \cdot)$ and $u(\ell, \cdot)$ are straight lines. Moreover, the inextensibility constraint implies that $|\o y|\le \ell$: indeed,
$$
|\o y|=|u(\ell,0)-u(0,0)|\le \int_0^\ell |\partial_1 u(x_1,0)|\,dx_1 =\ell.
$$

We thus consider as domain of the energy $E_\e$ the admissible class
\begin{align*}
\AA_\e = \big\{u\in W^{2,2}_{\rm iso}(S_\e;\R^3) &: 
u(0,x_2) = x_2,\, \d_1 u(0, x_2) = e_1,
\\
&\quad \, u(\ell,0)=\o y,\, \d_i u(\ell, x_2) = \o d_i,\, i=1,2
\big\},
\end{align*}
where all equalities are in the sense of traces.

If $|\o y|= \ell$, the midline $I\times\{0\}$ of the strip cannot deform and the cross sections cannot twist around the midline,
because otherwise the ``fibers'' $I\times\{\pm\e\}$ would get shorter or longer. In other words, if $|\o y|= \ell$, the whole strip cannot deform. 
This is proved in the next lemma.

\begin{lemma}\label{starrm}
Let $\e > 0$ and let $u\in W^{2,2}_{\rm iso}(S_\e;\R^3)$  be such that
$u(0) = 0$, $\D u = (e_1\, |\, e_2)$ on $\{0\}\times (-\e/2, \e/2)$,
$u(\ell, 0) = \ell e_1$, and $\D u$ constant on $\{\ell\}\times (-\e/2, \e/2)$. Then $u(x)=x_1e_1+x_2e_2$ on $S_{\e}$ .
\end{lemma}
\begin{proof}
Since $|u(0,0)-u(\ell,0)|=\ell=|(0,0)-(\ell,0)|$, by \cite[Lemma~2.4]{EberhardHornung} the gradient $\D u$ is constant on
the line segment $L=I\times\{0\}$.
Hence, $\D u = (e_1\, |\, e_2)$ on $L$, because $L$ intersects
$\{0\}\times (-\e/2, \e/2)$
and $\D u = (e_1\, |\, e_2)$ on this set. This implies, in particular, that $\D u = (e_1\, |\, e_2)$ on $\{\ell\}\times (-\e/2, \e/2)$.

Assume by contradiction that there exists $\hat x\in S_{\e}$ such that $\D u(\hat x)\neq (e_1\, |\, e_2)$.
Let $C_{\D u}$ be the set of points $x$ for which $\D u$ is constant in a neighborhood of $x$. 
If $\hat x\in S_{\e}\setminus C_{\D u}$, then there exists a unique line segment $[\hat x] \subset S_\e$ with
both endpoints on the boundary $\d S_{\e}$ such that the deformation gradient $\D u$ is constant on $[\hat x]$, see \cite{EberhardHornung}.
Since $[\hat x]$ must intersect $L$ or $\{0\}\times (-\e/2, \e/2)$ or $\{\ell\}\times (-\e/2, \e/2)$, we obtain a contradiction.
If $\hat x\in C_{\D u}$, then the boundary of the connected component of $C_{\D u}$ to which $\hat x$ belongs, contains at least a segment that intersects $L$ or $\{0\}\times (-\e/2, \e/2)$ or $\{\ell\}\times (-\e/2, \e/2)$, providing again a contradiction.
\end{proof}

By Lemma~\ref{starrm} and the above considerations we have that: if $|\o y|>\ell$, then $\AA_\e=\varnothing$;
if $|\o y|=\ell$, then $\AA_\e$ is either the empty set or reduces to the single map $u_0(x)=x_1e_1+ x_2e_2$,
according to $(\o d_1,\o d_2)$ being different or equal to $(e_1,e_2)$. Hence, the only non trivial case is $|\o y|<\ell$ (note that, if $|\o y|<\ell$,
then $\AA_\e\neq\varnothing$ for $\e>0$ small enough by Remark~\ref{nonempty} below).

Hereafter, we shall always assume that $|\o y|<\ell$.

\subsection{Change of variables}
We now change variables in order to rewrite the energy on the fixed domain
$$
S = I\times \big(-\tfrac12, \tfrac12 \big).
$$
We introduce the rescaled version $y: S\to\R^3$ of $u$ by setting
$$
y(x_1, x_2) = u(x_1, \e x_2).
$$
We have that
$$
\De y(x) = \D u(x_1, \e x_2),
$$
where the scaled gradient is defined by
$$
\De\, \cdot = (\d_1\cdot\, |\, \e^{-1}\d_2\,\cdot\,).
$$
In particular, if $u\in W^{2,2}_{\rm iso}(S_\e;\R^3)$, the map $y$ belongs to the space of {\it scaled isometries}
$$
W^{2,2}_{\rm iso, \e}(S;\R^3):= \big\{ y\in W^{2,2}(S;\R^3): \ |\partial_1 y|=\e^{-1} |\partial_2 y|=1, \
\partial_1 y\cdot \e^{-1} \partial_2 y=0 \text{\ in } S \big\},
$$
and 
$$
u\in \AA_\e\ \iff\ y\in \AA_\e^s,
$$
where the admissible class of scaled isometries $ \AA_\e^s$ is defined by
\begin{align}
\AA_\e^s = \big\{y\in W^{2,2}_{\rm iso,\e}(S;\R^3) &: \
y(0,0) = 0,
\,\d_1 y(0, x_2) = e_1,\ \d_2 y(0, x_2) = \e e_2, 
\nonumber
\\
&\quad y(\ell,0)=\o y,\, \d_1 y(\ell, x_2) = \o d_1,\ \d_2 y(\ell, x_2) = \e \o d_2\big\}. \label{Aepss}
\end{align}

We define the scaled unit normal to $y(S)$ by
$$
n_{y,\e} = \d_1 y\times \e^{-1}\d_2 y
$$
and the scaled second fundamental form of $y(S)$ by
$$
\Pi_{y,\e} =
\begin{pmatrix}
n_{y,\e}\cdot\d_1\d_1 y & \e^{-1}n_{y,\e}\cdot\d_1\d_2 y
\smallskip\\
\e^{-1}n_{y,\e}\cdot\d_1\d_2 y & \e^{-2}n_{y,\e}\cdot\d_2\d_2 y
\end{pmatrix},
$$
so that $\Pi_u(x_1, \e x_2) = \Pi_{y,\e}(x_1, x_2)$. Finally, we denote the scaled energy by
\begin{equation}\label{Jeps}
J_{\e}(y) = \int_S |\Pi_{y,\e}(x)|^2\, dx
\end{equation}
and we have $J_{\e}(y) = E_{\e}(u)$.

\subsection{Statement of the main results}

As $\e$ approaches zero, the convergence of the admissible deformations 
leads naturally, as shown in Lemma~\ref{compactness} below, to the admissible class
\begin{align}
\AA_0 = & \big\{ (y, r)\in W^{2,2}(I; \R^3)\times W^{1,2}(I; SO(3))  :\ r^T=(d_1|d_2|d_3),
\nonumber \\
& \quad y'=d_1,\, d_1'\cdot d_2=0,\, y(0) = 0,\, r(0) = I,\,
y(\ell) = \o y,\, 
r(\ell) = \o r\,\big\}. \label{AA0}
\end{align}

\begin{proposition}\label{A0-ne}
Assume  $|\o y|<\ell$. Then $\AA_0\ne\varnothing$.
\end{proposition}

\begin{proof}
The proposition follows from \cite[Proposition~3.1]{H2}. We provide here an explicit construction for the reader's convenience.

Given $\o y$ and $\o r$, with  $|\o y|<\ell$, we shall construct a pair $(y, r)\in \AA_0$.
To satisfy easily the constraint $d_1'\cdot d_2=0$, we shall build $y$   by ``glueing'' together straight curves and arcs
of circles.

We start by giving two definitions.

The pair $(y, r):[\ell^i,\ell^e]\to \R^3\times SO(3)$, where $\ell^i<\ell^e$, with 
 $y(t)=t e_1+\hat y$, $d_1(t)=e_1$, $d_2(t)=\cos(\alpha t+\beta) e_2 +\sin(\alpha t+\beta) e_3$, for some $\hat y\in\R^3$, $\alpha,\beta\in \R$, and $d_3(t)=d_1(t)\wedge d_2(t)$ shall be called a {\it straight frame} starting at $y(\ell^i)$, parallel to $e_1$, and that rotates $d_2(\ell^i)$ to  $d_2(\ell^e)$. Given three generic unit vectors $d, b^i, b^e$ such that $d\cdot b^i=d\cdot b^e=0$, and a point $x$, we can similarly define a  straight frame, starting at  $x$, parallel to $d$, and that rotates $b^i$ to  $b^e$.  Trivially, the straight frames satisfy the conditions  $y'=d_1$ and $d_1'\cdot d_2=0$.
 
 The pair $(y, r):[\ell^i,\ell^e]\to \R^3\times SO(3)$, where $\ell^i<\ell^e$, with 
 $y(t)=\sigma (\cos(\alpha+t/\sigma) e_2 +\sin(\alpha+t/\sigma) e_3)+\hat y$, for some $\hat y\in\R^3$, $\alpha,\sigma\in \R$, $\sigma\neq0$, $d_1(t)=y'$, $d_2(t)= e_1$, and $d_3(t)=d_1(t)\wedge d_2(t)$ shall be called  a {\it circular frame}, starting at $y(\ell^i)$, orthogonal to $e_1$, and that rotates $d_1(\ell^i)$ to  $d_1(\ell^e)$. If convenient, instead of specifying the starting point $y(\ell^i)$ we may specify the ending point $y(\ell^e)$. Given three generic unit vectors $b, d^i, d^e$ such that $b\cdot d^i=b\cdot d^e=0$, and a point $x$, we can similarly define a  circular frame, starting at $x$, orthogonal to $b$, and that rotates $d^i$ to  $d^e$. Clearly, the circular frames also satisfy the conditions  $y'=d_1$ and $d_1'\cdot d_2=-d_1\cdot d_2'=0$.
 
It is also convenient to denote by $S_{PQ}$ the segment whose endpoints are $P$ and $Q$.

Given these definitions we prove the proposition by first assuming $\o y\ne 0$.
Let 
$$
\delta:=\frac1{12}\,{\min\{|\o y|, \ell-|\o y|\}}.
$$
We define $(y,r)$ in several steps.
\begin{enumerate}
\item If $e_1$ is not orthogonal to the segment $S_{0\o y}$, let $(y^0,r^0):[0,\delta]\to \R^3\times SO(3)$ be a circular frame starting at $0$, orthogonal  to $e_2$, and that rotates $e_1$ to a unit vector ${\o d}_1^0:=(y^0)'(\delta)$ orthogonal to $S_{0\o y}$.
We  set $P^0:=y^0(\delta)$ and $\o d_2^0:=e_2$.
\item If $e_1$ is orthogonal to the segment $S_{0\o y}$, let $(y^0,r^0):\{0\}\to \R^3\times SO(3)$, with $y^0(0)=0$, $r^0(0)=I$, and let   $P^0:=0$, $\o d_1^0:=e_1$, and $\o d_2^0:=e_2$.
\item If $\o d_1\ne -\o d_1^0$, let $(y^\ell,r^\ell):[\ell-\delta,\ell]\to \R^3\times SO(3)$ be a map such that
$d_1^\ell(\ell-\delta)=-\o d_1^0$, $d_2^\ell(\ell-\delta)=\o d_2^0$, $y^\ell(\ell)=\o y$, and $r^\ell(\ell)=\o r$.
This can be achieved in the following way: let $a$ be a unit vector such that $a\cdot \o d_1=a\cdot \o d_1^0=0$; we glue together
a straight frame parallel to $-\o d_1^0$ that rotates $\o d_2^0$ to $a$, with a circular frame orthogonal to $a$ that rotates $-\o d_1^0$ to $\o d_1$, and finally with a straight frame ending at $\o y$, parallel to $\o d_1$ and that rotates $a$ to $\o d_2$.
We set $P^\ell:=y^\ell(\ell-\delta)$. 
\item If $\o d_1=-\o d_1^0$, let $(y^\ell, r^\ell):\{\ell\}\to \R^3\times SO(3)$, with $y^\ell(\ell)=\o y$, $r^\ell(\ell)=\o r$, and let   $P^\ell:=\o y$ and $\o d_2^\ell:=\o d_2$.
\item Let
$$
s^0:=\{P^0+t \o d_1^0: \ t\in (0,+\infty)\}
\quad \mbox{and}\quad
s^\ell:=\{P^\ell+t \o d_1^0: \ t\in (0,+\infty)\},
$$
and note that the distance between $s^0$ and $s^\ell$ is larger than $10\delta$ and smaller than $\ell-10\delta$.
Let now $Q^0\in s^0$ and $Q^\ell\in s^\ell$ be such that the segment $S_{Q^0Q^\ell}$ is orthogonal to
$\o d_1^0$. The length of the curve obtained by glueing together the curve $y^0$, the segment $S_{P^0Q^0}$, the segment $S_{Q^0Q^\ell}$,  the segment $S_{Q^\ell P^\ell}$, and the curve $y^\ell$, has a minimum value that is at most $\ell-6\delta$.
Therefore, we can choose $Q^0$ and $Q^\ell$ such that the total length of this curve
 is exactly equal to $\ell+(4-\pi)\delta$.
We denote by $d^p$ the unit vector parallel to $S_{Q^0Q^\ell}$ pointing towards~$Q^\ell$.
\item We are now in a position to define $(y,r)$. 
Let $b$ be a unit vector orthogonal to the plane containing the points $P^0$, $Q^0$, and $Q^\ell$ (note that $P^\ell$ belongs to this plane, too). We consider the following curves:
\begin{enumerate}

	\item[i)] Let $\eta^0$ be the distance between $P^0$ and $Q^0$ (note that $\eta^0>3\delta$) and let $\tilde \delta^i:=\delta$ if case (1) holds, and $\tilde \delta^i:=0$ if case (2) holds. Let $(y^i,r^i):[\tilde \delta^i, \tilde \delta^i+\eta^0-\delta] \to \R^3\times SO(3)$
be a straight frame starting at $P^0$, parallel to $\o d_1^0$, and that rotates $\o d_2^0$ to $b$.

	\item[ii)] Let $\eta^\ell$ be the distance between $P^\ell$ and $Q^\ell$ (note that $\eta^\ell>\delta$) and let $\tilde \delta^e:=\delta$ if case (3) holds, and $\tilde \delta^e:=0$ if case (4) holds. Let $(y^e,r^e):[\ell-\tilde \delta^e-\eta^\ell+\delta, \ell-\tilde \delta^e] \to \R^3\times SO(3)$ be a straight frame ending at $P^\ell$, parallel to $-\o d_1^0$, and that rotates $b$ to $\o d_2^0$ in case (3) and to $\o d_2^\ell$ in case~(4).
	
	\item[iii)] Let $(y^{ci},r^{ci}):[\tilde \delta^i+\eta^0-\delta, \tilde \delta^i+\eta^0-\delta+\pi\delta/2] \to \R^3\times SO(3)$ be a circular frame starting at $y^i(\tilde \delta^i+\eta^0-\delta)$, orthogonal to $b$, and that rotates $\o d_1^0$ to~$d^p$. 
    
	\item[iv)] Let $(y^{ce},r^{ce}):[\ell-\tilde \delta^e-\eta^\ell+\delta-\pi\delta/2,\ell-\tilde \delta^e-\eta^\ell+\delta] \to \R^3\times SO(3)$ be a circular frame ending at $y^e(\ell-\tilde \delta^e-\eta^\ell+\delta)$, orthogonal to $b$, and that rotates $d^p$ to $-\o d_1^0$.
	 
	\item[v)] Let $(y^p,r^p):[\tilde \delta^i+\eta^0-\delta+\pi\delta/2, \ell-\tilde \delta^e-\eta^\ell+\delta-\pi\delta/2] \to \R^3\times SO(3)$ be a straight frame, starting at $y^{ci}(\tilde \delta^i+\eta^0-\delta+\pi\delta/2)$, with $(r^p(t))^T=(d^p\,|\,b\,|\,d^p\wedge b)$.
    							
\end{enumerate}
We define $(y,r):[0,\ell]\to\R^3\times SO(3)$ as the function equal to $(y^0,r^0)$, $(y^i,r^i)$, $(y^{ci},r^{ci})$, $(y^p,r^p)$, $(y^{ce},r^{ce})$, $(y^e,r^e)$, and $(y^\ell,r^\ell)$ on the respective domains.
It is easy to check that $y$ and $r$ have the desired regularity and satisfy all the conditions in the definition on $\mathcal A_0$.
\end{enumerate}

If $\o y= 0$, one can consider first the map $(y^a,r^a):[3\ell/4,\ell]\to\R^3\times SO(3)$ defined by $(y^a(t), r^a(t))=((t-\ell) \o d_1,\o r)$  and then repeat the previous argument
with $3\ell/4$ in place of $\ell$, and with  $y^a(3\ell/4)=-\ell/4 \o d_1$ in place of $\o y$.
\end{proof}

The next lemma shows that the compactness result \cite[Lemma~2.1]{FrHoMoPa} remains true under our set of 
boundary conditions.

\begin{lemma}\label{compactness}
Let $(y_{\e})$ be a sequence of scaled isometries such that $y_\e\in \AA_\e^s$ for every $\e>0$ and
$$
\sup_{\e} J_{\e}(y_{\e}) < \infty.
$$
Then, up to a subsequence, there exists $(y,r)\in \AA_0$
such that
\begin{equation}
\label{conv}
y_{\e}\weak y \mbox{ in }W^{2,2}(S; \R^3),\qquad
\D_{\e} y_{\e} \weak (d_1\, |\, d_2)\mbox{ in }W^{1,2}(S; \R^{3\times 2}),
\end{equation}
and
$$
\Pi_{y_{\e},\e} \weak
\begin{pmatrix}
d'_1\cdot d_3 & d_2'\cdot d_3
\\
d_2'\cdot d_3 & \gamma
\end{pmatrix}
\mbox{ in }L^{2}(S; \R^{2\times 2}_{\sym})
$$
with $\gamma\in L^2(S)$. 
\end{lemma}

\begin{proof}[Proof of Lemma~\ref{compactness}]
The proof is an easy adaptation of that of \cite[Lemma~2.1]{FrHoMoPa} with 
slight modifications due to the presence of boundary conditions. 
In particular, estimates (3.2) in \cite{FrHoMoPa} imply that the sequence $(y_\e)$ is uniformly 
bounded in $W^{2,2}(S;\R^3)$ (without any additive constant). 
In addition, the boundary conditions for $y$ and $r$ are satisfied owing to \eqref{conv}, the continuity of traces,  
and the compact embedding of $W^{2,2}(S;\R^3)$ in 
$C(\overline{S})$ (in fact, in $C^{0,\lambda}(\overline{S})$ for every $\lambda\in(0,1)$),
which implies uniform convergence on $\overline{S}$ of weakly converging sequences in $W^{2,2}(S;\R^3)$.
More precisely, the conditions $y(0)=0$ and $y(\ell)=\o y$ follow from passing to the limit in $y_\e(0,0) =0$ and
$y_\e(\ell,0)=\o y$, respectively, using that $y_\e\weak y$ in $W^{2,2}(S;\R^3)$, hence uniformly on $\overline{S}$.
The equality $y' = r^T e_1$ is a consequence of the definition of $r$ and \eqref{conv}. The condition $y'(0)=e_1$ follows from $\d_1 y_\e(0, x_2) = e_1$, the fact that $\d_1 y_\e\weak y'$ in $W^{1,2}(S;\R^3)$, and the continuity of the trace. Similarly, the condition $d_2(0)=e_2$ follows from $\d_2 y_\e(0, x_2) = \e e_2$, the fact that $\frac{\d_2 y_\e}{\e}\weak d_2$ in $W^{1,2}(S;\R^3)$, and the continuity of the trace.
This implies that $d_3(0)=d_1(0)\times d_2(0)=e_3$, hence $r(0) = I$. Analogously, one deduces that $r(\ell)=\o r$.
\end{proof}

The following theorem is the main result of the paper. It proves that the functionals $J_\e$ defined in \eqref{Jeps} with domain $\AA_\e^s$ (see \eqref{Aepss}) 
$\Gamma$-converge to the functional
$$
E(y,d_1,d_2,d_3):=\int_{0}^{\ell}\o Q(d_1'\cdot d_3, d_2'\cdot d_3)\,ds
$$
with domain $\AA_0$ (see \eqref{AA0}), where $\o Q$ defined in \eqref{barQ}.

\begin{theorem}\label{Gamma}
As $\e\to 0$, the functionals $J_\e$, with domain $\AA^s_\e$,  $\Gamma$-converge to the limit functional $E$,  with domain $\AA_0$, in the following sense: 
\begin{enumerate}
\item[(i)] {\rm (liminf inequality)} for every $(y,r)\in \AA_0$ and every sequence $(y_{\e})$ such that $y_\e\subset \AA_\e^s$ for every $\e>0$, $y_{\e}\weak y$ in $W^{2,2}(S; \R^3)$, and $\D_{\e} y_{\e} \weak (d_1\, |\, d_2)$ in $W^{1,2}(S; \R^{3\times 2})$,
we have that
$$
\liminf_{\e\to 0} J_\e(y_\e)\ge E(y,d_1,d_2,d_3);
$$
\item[(ii)] {\rm (recovery sequence)} for every $(y,r)\in \AA_0$ there exists a sequence $(y_{\e})$ such that
$y_{\e}\in \AA^s_\e$ for every $\e>0$, 
$y_{\e}\weak y$ in $W^{2,2}(S; \R^3)$, $\D_{\e} y_{\e} \weak (d_1\, |\, d_2)$  in $W^{1,2}(S; \R^{3\times 2})$, and
$$
\limsup_{\e \to 0} J_\e(y_\e)\le E(y,d_1,d_2,d_3).
$$
\end{enumerate}
\end{theorem}

\begin{remark}\label{nonempty}
By combining Proposition~\ref{A0-ne} and Theorem~\ref{Gamma}\,--\,(ii) we deduce, in particular, that $\AA^s_\e\neq\varnothing$, hence $\AA_\e\neq\varnothing$
for $\e>0$ small enough. For the boundary conditions of a M\"obius band an explicit construction was provided by Sadowsky in \cite{Sadowsky}, see also \cite{HF2015b}.
\end{remark}

The liminf inequality can be proved exactly as in \cite{FrHoMoPa}. We postpone to Section~\ref{RS} the proof of the existence of a recovery sequence, which is based on the approximation results of the next section.

\section{Smooth approximation of infinitesimal ribbons}\label{UB}

The first step in the construction of the recovery sequence
consists in showing that it is enough to
construct ribbons with finite width 
starting from well-behaved
infinitesimally narrow ribbons.
At the level of the infinitesimal
ribbons we perform several approximation steps in which
we iteratively approximate and correct
the approximating sequences on ever finer scales.
It is essential to correct
the boundary conditions at the end of each step.
The procedure is not trivial, because
the correction process could spoil other essential properties.
For this reason, one part of the modifications aims
at making the sequence robust enough to be stable under the corrections.

We set $\ell=1$ and $I = (0, 1)$. The letter $I$ will also denote the identity matrix.
All results in this section remain true with obvious changes
for intervals of arbitrary length $\ell\in (0, \infty)$.
We define $\A\subset\R^{3\times 3}$
to be the span of $e_1\otimes e_3 - e_3\otimes e_1$
and $e_2\otimes e_3 - e_3\otimes e_2$.
For a given $A\in L^2(I;\A)$ we define $\ra: I\to SO(3)$ to be the solution
of the ODE system
$$
\ra' = A\,\ra \quad \text{ in } I
$$
with initial condition $\ra(0) = I$.
We will call a map $A\in L^2(I;\A)$ {\it nondegenerate}
on a measurable set $J\subset I$ if $J\cap\{A_{13}\neq 0\}$ has positive measure.
When the set $J$ is not specified, it is understood
to be $J = I$.

For all $A\in L^2(I;\A)$ we define
$$
\Gamma_{\!A} := \int_0^1 \ra^T(t)e_1\, dt.
$$
We fix some nondegenerate $A^{(0)}\in L^2(I;\A)$
and set $\o r = r_{\!A^{(0)}}(1)$ and 
$\o\Gamma = \Gamma_{\!A^{(0)}}$.
A map $A\in L^2(I;\A)$ is said to be {\it admissible} if it satisfies 
\begin{equation}\label{defvonM}
\ra(1) = \o r \quad \text{ and } \quad
\Gamma_{\!A} = \o\Gamma.
\end{equation}

Note that, if $(y,r)\in\AA_0$ (see \eqref{AA0}), then $r=\ra$ with $A$ given by
$$
A= \begin{pmatrix}
0 & 0 & d_1'\cdot d_3
\\
0 & 0 & d_2'\cdot d_3
\\
-d_1'\cdot d_3 & -d_2'\cdot d_3 & 0
\end{pmatrix}.
$$ 
In particular, $A$ is admissible in the sense of \eqref{defvonM} with $\o r$ given by the boundary condition at $\ell$ and
$\o\Gamma=\o y$.

For $M\in\R^{2\times 2}_{\sym}$ we define $A_M\in\A$ by setting
$(A_M)_{13} = M_{11}$ and $(A_M)_{23} = M_{12}$.
For $M\in L^2(I;\R^{2\times 2}_{\sym})$ we introduce the functional
$$
\FF(M):= \int_0^1( |M|^2 + 2|\det M|)\, dt.
$$
In \cite[Lemma~3.1]{FrHoMoPa} it has been proved that for every $M\in L^2(I;\R^{2\times 2}_{\sym})$ there exists a sequence $(M_n)\subset L^2(I;\R^{2\times 2}_{\sym})$
such that $\det M_n=0$ for every $n$, $M_n\weak M$ weakly in $L^2(I;\R^{2\times 2}_{\sym})$, and $\FF(M_n)\to\FF(M)$.
The main purpose of this section is to prove the following refinement of this result.

\begin{proposition}\label{haupt}
Let $M\in L^2(I;\R^{2\times 2}_{\sym})$
be such that $M_{11}\neq 0$ on a set of positive measure and $A_M$ is admissible.
Then
there exist $\la_n\in C^{\infty}(\o I)$
and $p_n\in C^{\infty}(\o I;\S^1)$ 
such that, setting $M_n:=\la_n p_n\otimes p_n$, we have
\begin{enumerate}[(i)]
\item $p_n\cdot \u e_1 > 0$ everywhere on $\o I$, 
$p_n = \u e_1$ near $\d I$, and $A_{M_n}$ is admissible for every~$n$;
\item \label{haupt-ii} 
$M_n\weak M$ weakly
in $L^2(I;\R^{2\times 2}_{\sym})$;
\item $\FF(M_n) \to \FF(M)$.
\end{enumerate}
\end{proposition}

The next lemma is the key tool that allows us to correct the boundary conditions at each approximation step.

\begin{lemma}\label{korrektur}
Let $A\in L^2(I;\A)$ and assume
that there is a set $J\subset I$ of positive measure
such that $A$ is nondegenerate on $J$. Then every $L^2$-dense subspace $\t E$
of
$$
\{\h A\in L^2(I;\A) : \ \h A=0 \text{ a.e.\ in }I\setminus J\}
$$
contains a finite dimensional subspace $E$
such that, whenever $A_n\in L^2(I; \A)$ 
converge to $A$ weakly in $L^2(I; \A)$, then
there exist $\h A_n\in E$ converging to zero in $E$
and such that
$$
r_{\! A_n + \h A_n}(1) = \ra(1) \quad \text{ and } \quad\Gamma_{\! A_n + \h A_n} = \Gamma_{\!A}
$$
for every $n$ large enough.
\end{lemma}

\begin{proof}
By \cite[Proposition~3.3]{framed}, if $A$ is nondegenerate
on $J$, then it is also nondegenerate in the
sense of \cite[Definition~3.1]{framed}. Hence, the lemma follows directly from 
\cite[Theorem~3.2]{framed}.
\end{proof}

\begin{remark}
If $\t E\subset L^\infty(I; \A)$, the sequence $\h A_n$ provided by the lemma converges to zero uniformly,
since $E$ is finite dimensional and all norms are topologically equivalent in finite dimension.
We shall use this remark several times in the following.
\end{remark}

We now prove several approximation results for nondegenerate and admissible functions
in $L^2(I;\A)$, that preserve admissibility.

\begin{lemma}\label{nochmal0}
Let $A\in L^2(I;\A)$ be nondegenerate and admissible. Then there exist
admissible $A^{(n)}\in L^2(I;\A)$ such that
$A^{(n)}\to A$ strongly in $L^2(I;\A)$ and
$A^{(n)}_{13}A^{(n)}_{23}\neq 0$ on a set of positive measure, independent of $n$.
\end{lemma}

\begin{proof}
We may assume that $A_{13}A_{23} = 0$ almost everywhere,
since otherwise there is nothing to prove.
Since $A$ is nondegenerate, there exist two disjoint sets
$J_1$, $J_2$ of positive measure on which $A_{13}$ does not vanish.
Let $\t A^{(n)} : I\to\A$ be equal to $A$ except  on $J_2$ where we set $\t A_{23}^{(n)} = \frac{1}{n}$.
Since $A_{23} = 0$ on $J_2$, we see that $\t A^{(n)}$
converges to $A$ strongly in $L^2(I;\A)$.

Let $\t E$ be the set of maps in $L^{\infty}(I; \A)$
that vanish a.e.\ on $I\setminus J_1$. By Lemma~\ref{korrektur}
there exist $\h A_n\in\t {E}$ converging to zero uniformly
and such that $A^{(n)} := \t A^{(n)} + \h A^{(n)}$ is admissible.
By construction we have $A^{(n)}_{13}A^{(n)}_{23}=\frac{1}{n}A_{13}\neq 0$
on $J_2$, and $A^{(n)}\to A$ strongly in $L^2(I;\A)$.
\end{proof}

For $c > 0$ we set 
$$
\A_c := \{A\in\A :\ |A_{13}|\geq c\mbox{ and }|A_{23}|\geq c\}.
$$

\begin{lemma}\label{nochmal1}
Let $A\in L^2(I; \A)$ be nondegenerate and admissible. Then there exist
admissible $A^{(n)}\in L^2(I;\A_{1/n})$ such that
$A^{(n)}\to A$ strongly in $L^2(I;\A)$.
\end{lemma}

\begin{proof}
In view of Lemma~\ref{nochmal0}
we may assume, without loss of generality,
that $A_{13}A_{23}$ differs from zero on a set $J_{\infty}$ of positive measure.
For $k\in\NN$ define
$$
J_k := \left\{
t\in I : \ |A_{13}(t)|\geq\frac{1}{k}\mbox{ and }|A_{23}(t)|\geq\frac{1}{k}
\right\}.
$$
Clearly
$J_k\uparrow J_{\infty}$,
as $k\to +\infty$. 
Since $J_{\infty}$ has positive measure,
there exists $K\in\NN$ such that $J_K$ has positive measure.
By definition, $A$ is nondegenerate on $J_K$.
Denote by $\t E$ the set of
maps in $L^{\infty}(I;\A)$ that vanish a.e.\ in $I\setminus J_K$.
Now define
$\t A^{(n)} : I\to\A$
as follows: for $(i, j) = (2,3)$ and $(i, j) = (1,3)$ 
set
$$
\t A^{(n)}_{ij}(t) = 
\begin{cases}
A_{ij}(t) & \mbox{ if }|A_{ij}(t)| > \frac{1}{n},
\\
\frac{1}{n} &\mbox{ if }|A_{ij}(t)|\leq\frac{1}{n}.
\end{cases}
$$
Since $|\t A^{(n)} - A|\leq\frac{2}{n}$, we have that
$\t A^{(n)}\to A$ strongly in $L^2(I;\A)$. Hence,
by Lemma~\ref{korrektur} there exist a finite dimensional
subspace $E$ of $\t E$ and some $\h A^{(n)}\in E$ such that
$\h A^{(n)}\to 0$ uniformly and $A^{(n)}: = \t A^{(n)} + \h A^{(n)}$
is admissible.

Almost everywhere on $I\setminus J_K$ we have
$$
A^{(n)} = \t A^{(n)}\in\A_{1/n},
$$
whereas on $J_K$, for large $n$, we have $\t A^{(n)} = A$ and thus,
$$
|A^{(n)}_{13}| \geq |A_{13}| - \|\h A^{(n)}\|_{L^{\infty}} \geq\frac{1}{2K}.
$$
The same holds for $A^{(n)}_{23}$.  Hence $A^{(n)}\in L^2(I; \A_{1/n})$ for any $n$ large enough. 
Moreover,  $A^{(n)} = \t A^{(n)} + \h A^{(n)}\to A$ strongly in $L^{\infty}(I;\A)$
and the proof is concluded.
\end{proof}

We denote by $\P(I; X)$ the piecewise constant functions
from $I$ into the set $X$, i.e., $f\in\P(I; X)$ if there
exists a finite covering of $I$ by disjoint nondegenerate intervals on each of
which $f$ is constant. Moreover, we set $\P(I):=\P(I; \R)$.

\begin{lemma}\label{nochmal2}
Let $A\in L^2(I; \A)$ be nondegenerate and admissible. Then there exist
$c_n > 0$ and admissible $A^{(n)}\in\P(I; \A_{c_n})$ such that
$A^{(n)}\to A$ strongly in $L^2(I; \A)$.
\end{lemma}

\begin{proof}
In view of Lemma~\ref{nochmal1} we may assume that 
$A$ take values in $\A_{2\e}$ for some $\e > 0$.
Hence  there exist
$\t A^{(n)}\in\P(I; \A_{\e})$ which converge to $A$ strongly in $L^2(I; \A)$.
Indeed, for $i=1,2$, we can write $A_{i3}$ as difference of the positive and negative parts, $A_{i3}^+$ and $A_{i3}^-$, which are not both smaller than $\e$. It is well known that there exist increasing sequences of simple measurable functions $S_{i3}^{\pm(n)}$ such that $\e\le S_{i3}^{\pm(n)}\le A_{i3}^{\pm}$ and
$S_{i3}^{\pm(n)}\to A_{i3}^\pm$ a.e.\ in $I$. Since the sets where the simple functions $S_{i3}^{\pm(n)}$ are constant are measurable, they can be approximated in measure from inside 
by finite unions of intervals. Hence, the same result holds for $S_{i3}^{\pm(n)}\in\P(I; [\e,+\infty))$.  By Lebesgue's theorem the sequences $S_{i3}^{(n)}:=S_{i3}^{+(n)}-S_{i3}^{-(n)}$ strongly converge in $L^2(I; \A)$ to $A_{i3}$.
Applying Lemma~\ref{korrektur} with
$\t E = \P(I; \A)$, we find $\h A^{(n)}\in\P(I; \A)$ converging to
zero uniformly and such that $A^{(n)} := \t A^{(n)} + \h A^{(n)}$
is admissible for $n$ large enough. Clearly,
$A^{(n)}\in\P(I; \A_{\e/2})$ for $n$ large.
\end{proof}

For $A\in\A$ we define
$$
\gamma_A := \frac{A_{23}^2}{A_{13}},
$$
with $\gamma_A := +\infty$ if $A_{13} = 0$ and $A_{23}\neq 0$, whereas
$\gamma_A := 0$ if $A_{13} = A_{23} = 0$.

\begin{lemma}\label{schrittz}
For all nondegenerate and admissible $A\in L^2(I;\A)$
and all $\gamma\in L^2(I)$ there exist
$\la_n\in\P(I)$ and $p_n\in\P(I; \S^1)$ with $\la_n\ne0$, 
$p_n\cdot \u e_1 > 0$ on $I$ and $p_n = \u e_1$ near $\d I$,
such that 
$$
A^{(n)} = A_{\la_n p_n\otimes p_n}
$$
is admissible, $(A^{(n)}, \gamma_{A^{(n)}})\weak (A, \gamma)$ weakly in $L^2(I;\A\times\R)$,
and 
\begin{equation}\label{limF}
\FF
\begin{pmatrix}
A^{(n)}_{13} & A^{(n)}_{23}
\\
A^{(n)}_{23} & \gamma_{A^{(n)}}
\end{pmatrix} \,
\to \, \FF
\begin{pmatrix}
A_{13} & A_{23}
\\
A_{23} & \gamma
\end{pmatrix}.
\end{equation}
\end{lemma}
\begin{proof}
We may assume without loss of generality that $\gamma\in\P(I)$.
In fact, if the lemma is true for all $\gamma\in\P(I)$, then,
by approximation and a standard diagonal procedure, it is true also for arbitrary $\gamma\in L^2(I)$.
For the same reason, in view of Lemma~\ref{nochmal2}
we may assume that $A\in\P(I;\A_c)$ for some $c > 0$.
Hence the map 
$$
M := \begin{pmatrix}
A_{13} & A_{23}
\\
A_{23} & \gamma
\end{pmatrix}
$$
is piecewise constant and never zero.
Since $M_{12}$ is bounded away from zero, the eigenvectors
of $M$ have components along $\u e_1$ and $\u e_2$ bounded away from zero.
We choose $a(x)\in\S^1$ among both eigenvectors and among both signs
to be an eigenvector
of $M(x)$ for which $a(x){\,\cdot\,} \u e_1$ is maximal; in particular, $a{\,\cdot\,} \u e_1\geq c$
for a positive constant~$c$.
Define
$$
\t a = 
\begin{cases}
a^{\perp} &\mbox{ if }a\cdot \u e_2 < 0,
\\
-a^{\perp} &\mbox{ if }a\cdot \u e_2 > 0,
\end{cases}
$$
where $a^{\perp} = (-a_2, a_1)$. 
Then $\t a$ is piecewise constant and $\t a\cdot \u e_1 > 0$.

Denote by $\la_1$ the eigenvalue corresponding to $a$, and by $\la_2$
the other one. Like $M$ itself, the function
$$
\La := |\la_1| + |\la_2|
$$
is piecewise constant and never zero.
Moreover, there 
exist piecewise constant functions $\theta : I\to [0, 1]$
and $\sigma^{(i)} : I\to\{-1, 1\}$ such that 
the following spectral decomposition holds:
$$
M = \Lambda\left(\sigma^{(1)}\theta a\otimes a + \sigma^{(2)} 
(1 - \theta) \t a\otimes \t a \right).
$$

Let $\chi_n : I\to \{0, 1\}$ be such that $\chi_n\weak\theta$
weakly$^*$ in $L^{\infty}(I)$ and define 
\begin{align*}
\sigma_n &:= \begin{cases}
\sigma^{(1)}\chi_n + \sigma^{(2)} (1-\chi_n)&\mbox{ on }(\frac{1}{n},1-\frac{1}{n}),
\\
1&\mbox{ elsewhere, }
\end{cases}
\end{align*}
and 
$$
\t p_n := 
\begin{cases}
\chi_n a + (1-\chi_n)\t a &\mbox{ on }(\frac{1}{n},1-\frac{1}{n}),
\\
e_1 &\mbox{ elsewhere. }
\end{cases}
$$
Finally, let 
$$
\t M_n := \sigma_n \La \t p_n\otimes \t p_n.
$$
Since we can write
$$
\t M_n=\begin{cases}
\Lambda\left(\chi_n\sigma^{(1)}a\otimes a + (1-\chi_n)\sigma^{(2)} 
 \t a\otimes \t a \right) &\mbox{ on }(\frac{1}{n},1-\frac{1}{n}),
\\
\Lambda \u e_1\otimes \u e_1  &\mbox{ elsewhere, }
\end{cases}
$$
we have that
 $\t M_n$ converges to $M$ weakly$^*$ in $L^{\infty}(I;\R^{2\times 2}_{\sym})$.
Moreover,
\begin{equation}\label{Fequal}
\int_0^1 |\t M_n|^2\, dt = \int_0^1\La^2\, dt = \int_0^1( |M|^2 + 2|\det M|)\, dt \qquad
\mbox{ for all }n.
\end{equation}

Let $\t A^{(n)} := A_{\t M_n}$. Since $A = A_M$, we have that $\t A^{(n)}$ converges to $A$ weakly$^*$ in $L^{\infty}(I;\A)$.
Moreover, $\gamma_{\t A^{(n)}}=(\t M_n)_{22}$ since $\det \t M_n=0$, hence $\gamma_{\t A^{(n)}}$  weakly$^*$ converges to $M_{22}=\gamma$ 
in $L^{\infty}(I)$. Condition \eqref{Fequal} rewrites as
\begin{equation}\label{FequalA}
\FF
\begin{pmatrix}
A^{(n)}_{13} & A^{(n)}_{23}
\\
A^{(n)}_{23} & \gamma_{A^{(n)}}
\end{pmatrix}
=
\FF
\begin{pmatrix}
A_{13} & A_{23}
\\
A_{23} & \gamma
\end{pmatrix} \qquad
\mbox{ for all }n.
\end{equation}

We now modify $\t A^{(n)}$ in such a way to make it admissible.
Let $J\subset(\frac{1}{4}, \frac{3}{4})$
be an open interval on which $M$ is constant.
Denote by $\t E$ the set of maps in $\P(I; \A)$ that vanish a.e.\ in $I\setminus J$.
Since $A$ is nondegenerate on $J$, by Lemma~\ref{korrektur}
there is a finite dimensional subspace $E\subset\t E$
and $\h A^{(n)}\in E$ converging to zero uniformly such that
$A^{(n)} := \t A^{(n)} + \h A^{(n)}$ is admissible.
Clearly, $A^{(n)}$ is piecewise constant and $A^{(n)}$ converges to $A$ weakly$^*$ in $L^{\infty}(I;\A)$. 

We claim that $\gamma_{A^{(n)}}$ converges to $\gamma$  weakly$^*$
in $L^{\infty}(I)$. It is enough to show that
\begin{equation}\label{gamma-u}
\gamma_{A^{(n)}} - \gamma_{\t A^{(n)}} \to 0
\quad \mbox{ uniformly on }I.
\end{equation}
We first note that
there is a constant $c>0$ such that for all $n$
\begin{equation}
\label{einsicht-1}
|\t A^{(n)}_{13}| = |(\t M_n)_{11}|
= |\sigma_n\La (\t p_n\cdot e_1)^2|\geq c
\quad\mbox{ on }I. 
\end{equation}
The same is true (with a smaller constant) for $A^{(n)}_{13}$ for large $n$, 
since $\h A^{(n)}\to 0$ uniformly. 
Therefore, using that $\t A^{(n)}$ and $A^{(n)}$ are uniformly bounded, we obtain
\begin{eqnarray*}
|\gamma_{A^{(n)}}-\gamma_{\t A^{(n)}}|&=&\left|\frac{(A_{23}^{(n)})^2}{A_{13}^{(n)}}-\frac{(\t A^{(n)}_{23})^2}{\t A^{(n)}_{13}}\right|=
\left|\frac{(A_{23}^{(n)})^2 \t A^{(n)}_{13} - (\t A^{(n)}_{23})^2 A_{13}^{(n)}}{A_{13}^{(n)}\t A^{(n)}_{13}}\right|\\
&\le&
\frac{1}{c^2}\Big(\big|\t A^{(n)}_{13} \big|\big|{(A_{23}^{(n)})^2- (\t A^{(n)}_{23})^2 \big|+\big|(\t A^{(n)}_{23})^2\big|\big|\t A^{(n)}_{13} - A_{13}^{(n)}}\big|\Big)\\
&\le&
C\big|{\h A_{23}^{(n)}}|+C|\h A_{13}^{(n)}|.
\end{eqnarray*}
which implies \eqref{gamma-u}.

Equation \eqref{limF} follows from \eqref{FequalA}, \eqref{gamma-u}, and the uniform convergence of $\h A^{(n)}$.

It remains to show that $A^{(n)}$ is of the form $A_{\la_n p_n\otimes p_n}$ for some $\la_n$ and $p_n$ satisfying the desired properties. 
Keeping in mind that $A^{(n)}_{13}\neq 0$ due to \eqref{einsicht-1},
we define 
\begin{equation}
\label{defpn}
p_n :=
\alpha_n
\begin{pmatrix}
A^{(n)}_{13}
\\
A^{(n)}_{23}
\end{pmatrix},
\end{equation}
where
$$
\alpha_n:=\frac{\sgn A^{(n)}_{13}}{\left(( A^{(n)}_{13})^2 + (A^{(n)}_{23})^2\right)^{1/2}},
$$
and $\la_n := A^{(n)}_{13} + \gamma_{A^{(n)}}$.
It is easy to check that $A^{(n)} = A_{\la_n p_n\otimes p_n}$,
and $\la_n$ and $p_n$ have all the stated properties. 
In particular, $p_n=\t p_n$ on $I\setminus J$, hence $p_n = \u e_1$ near $\d I$, and by a consequence of \eqref{einsicht-1}
$$
|\la_n|=\left|A^{(n)}_{13}+\frac{(A^{(n)}_{23})^2}{A^{(n)}_{13}}\right|=
\left|\frac{(A^{(n)}_{13})^2+(A^{(n)}_{23})^2}{A^{(n)}_{13}}\right|\ge|A^{(n)}_{13}|\ge c
$$
for a suitable constant $c>0$.
\end{proof}

We are now in a position to prove the main result of this section, namely Proposition~\ref{haupt}.

\begin{proof}[Proof of Proposition~\ref{haupt}] 
Let us assume
that $M = \la p\otimes p$,
where
$p : I\to\S^1$ and $\la: I\to\R\setminus\{0\}$ are piecewise constant, 
$p\cdot \u e_1 > 0$ on $I$ and $p = \u e_1$ near $\d I$.
In view of Lemma~\ref{schrittz}
this constitutes no loss of generality, provided we prove
strong rather than weak convergence in Proposition~\ref{haupt}\,--\,\eqref{haupt-ii}. 
Indeed, since $M_{11}\ne0$ on a set of positive measure, then $A_M$ is nondegenerate
and we can apply Lemma~\ref{schrittz} to $A=A_M$ and $\gamma=M_{22}$. Then there exists a sequence
$M^{(n)}=\lambda_np_n\otimes p_n$ with $0\ne\la_n\in\P(I)$ and $p_n\in\P(I;\S^1)$ as in Lemma~\ref{schrittz}. 
To each element of this sequence we can apply the version of Proposition~3.1
that we are going to prove obtaining a sequence $M^{(n)}_k:=\lambda^k_np_n^k\otimes p_n^k$ such that $M^{(n)}_k$ converges to $M^{(n)}$ strongly in $L^2(I;\R^{2\times 2}_{\sym})$ and 
$\F(M^{(n)}_k)\to\F(M^{(n)})$, as $k\to\infty$. The required sequence is then given by $M^{(n)}_{k_n}$, where 
$(k_n)$ is an increasing sequence such that $\|M^{(n)}_{k_n}-M^{(n)}\|_{L^2}+|\F(M^{(n)}_k)-\F(M^{(n)})|<\frac{1}{2^n}$ for every $n$.

Let us write $p(t)={\rm e}^{i\theta(t)}$ with argument $\theta\in\P(I)$. Since $p$ is a Lipschitz continuous function of $\theta$, we can mollify the argument to obtain smooth
$\t p_n : I\to\S^1$ converging to $p$ boundedly in measure, that is,  $\t p_n$ converges in measure and 
$\sup_n\|\t p_n\|_{L^\infty}<+\infty$. Note that for every $n$ large enough we have $\t p_n = \u e_1$ near $\d I$,
as well as
\begin{equation}\label{pnba0}
\t p_n\cdot \u e_1\geq c
\end{equation}
for some constant $c > 0$.
The latter follows from the same property for $p$, which is 
stable under mollification of the argument.
By mollifying  $\la$ we obtain smooth $\t\la_n$ converging to $\la$
boundedly in measure such that
\begin{equation}\label{lnba0}
|\t\la_n|\ge c
\end{equation} 
for some constant $c>0$ and for every $n$ large enough, since a similar inequality holds for $\lambda$.
Observe that both $\t p_n$ and $\t\la_n$ are well-defined
and smooth up to the boundary of~$I$.

Let us define 
$
\t M_n := \t\la_n \t p_n\otimes \t p_n.
$
Since $\t M_n\to M$ boundedly in measure, we have that
$A_{\t M_n}\to A_M$ and $\gamma_{A_{\t M_n}}=({\t M_n})_{22}\to M_{22}=\gamma_{A_M}$ in the same sense.

We now modify $A_{\t M_n}$ in such a way to make it admissible.
Let $J\subset (1/4, 3/4)$ be a nondegenerate open interval
on which $\la$ and $p$ are constant.
Let $\t E = C_0^{\infty}(J; \A)$ be the space of smooth functions with compact support. Since $A_M$ is admissible and nondegenerate on $J$, by  Lemma~\ref{korrektur}
there exists a finite dimensional subspace $E\subset\t E$
and some $\h A^{(n)}\in E$ converging to zero uniformly
such that $A^{(n)} := \h A^{(n)} + A_{\t M_n}$ are admissible.
Clearly, $A^{(n)}$ is smooth up to the boundary and $A^{(n)}\to A_M$ boundedly in measure.

We claim that $\gamma_{A^{(n)}}\to\gamma_{A_M}$
boundedly in measure. 
First of all,  by  \eqref{pnba0} and \eqref{lnba0} there exists a constant $c>0$ such that  $|({\t M_n})_{11}|=|\t\la_n||\t p_n\cdot \u e_1|^2\ge c$. This implies that also
$A_{13}^{(n)}=({\t M_n})_{11}+\h A_{13}^{(n)}$ is bounded away from zero for $n$ large enough, because $\h A_{13}^{(n)}$ converges uniformly to $0$. Therefore, we can argue as in the proof of Lemma~\ref{schrittz}
and show that $\gamma_{A^{(n)}}-\gamma_{A_{\t M_n}}\to 0$ uniformly. Since $\gamma_{A_{\t M_n}}\to\gamma_{A_M}$ 
boundedly in measure, this proves the claim.

We now define $p_n$ as in \eqref{defpn} and $\la_n := A_{13}^{(n)} + \gamma_{A^{(n)}}$, so that $A^{(n)} = A_{\la_n p_n\otimes p_n}$.  
Since $A^{(n)}$ is smooth up to the boundary and $A_{13}^{(n)}$ is bounded away from zero for $n$ large,
the functions $p_n$ and $\la_n$ are smooth and $p_n{\,\cdot\,} \u e_1 > 0$ on $\o I$. By construction $p_n=\t p_n$ on $I\setminus J$, hence $p_n = \u e_1$ near $\d I$.

Finally, since $(\la_n p_n\otimes p_n)_{22}=\gamma_{A^{(n)}}$, we have that
$\la_n p_n\otimes p_n\to M$ boundedly in measure, hence
strongly in $L^2(I;\R^{2\times 2}_{\sym})$. This implies condition~(iii).
\end{proof}

\section{The recovery sequence}\label{RS}

In this section we prove part (ii) of Theorem~\ref{Gamma}, namely the existence of a recovery sequence.

\begin{proof}[Proof of Theorem~\ref{Gamma}\,--\,(ii)]
Let $(y,r)\in \AA_0$.
We set 
$$
M:= \left( \begin{array}{cc}
d_1'\cdot d_3 & d_2'\cdot d_3 \\
d_2'\cdot d_3 & \gamma
\end{array}\right),
$$
where $\gamma\in L^2(I)$ is such that
$$
\overline Q(d_1'\cdot d_3, d_2'\cdot d_3)=|M|^2+2|\det M| \qquad \text{a.e.\ in }I.
$$
Such a $\gamma$ can indeed be chosen measurable. Moreover, $\gamma\in L^2(I)$ because by minimality,
comparing $M$ to the same matrix with $0$ in place of $\gamma$,
we have
$$
\gamma^2 \leq |M|^2 + 2|\det M| \leq M_{11}^2 + 4 M_{12}^2  \qquad \text{a.e.\ in }I,
$$
and the right-hand side is in $L^1(I)$. 

Since $|\o y|<\ell$, the director $d_1$ cannot be constant, thus $M_{11}=d_1'\cdot d_3\neq 0$ on a set of positive measure. Moreover, the boundary conditions satisfied by $(y,r)$ guarantee that $A_M$ is admissible in the sense of \eqref{defvonM} with respect to the data $\o r$ and $\o\Gamma=\o y$.
By Proposition~\ref{haupt} there exist $\la_j\in C^{\infty}(\o I)$
and $p_j\in C^{\infty}(\o I; \S^1)$ 
such that $M^j := \la_j p_j\otimes p_j$ satisfies
\begin{enumerate}[(i)]
\item $p_j\cdot \u e_1 > 0$ everywhere on $\o I$, $p_j = \u e_1$ near $\d I$, and $A_{M^j}$ is admissible;
\item $M^j\weak M$ weakly in $L^2(I;\R^{2\times 2}_{\sym})$;
\item there holds
$$
\lim_{j\to\infty}\int_0^\ell |M^j|^2\,dt = \int_0^\ell( |M|^2 + 2|\det M|)\, dt.
$$
\end{enumerate}

Let $r_j : I\to SO(3)$ be the solution of the Cauchy problem
\begin{equation}
\label{odde}
\begin{cases}
(r_j)'= A_{M^j} r_j & \text{ in } I,\\
r_j(0)=I.
\end{cases}
\end{equation}
Since $M_j$ is smooth, so is $r_j$. Moreover, since $A_{M^j}$ is admissible, we have that $r_j(\ell)=\o r$ and
$$
\int_0^\ell r_j^T(t)e_1\, dt=\o y.
$$%
For $t\in I$ we define
$$
d^j_k(t) := r_j^T(t) e_k \quad \text{ for } k=1,2,3,  \qquad y^j(t):=\int_0^{t} d^j_1(s)\,ds,
$$
and we observe that $y^j(\ell)=\o y$.
One can show that $y^j\weak y$ weakly in $W^{2,2}(I;\R^3)$; see, for instance, the proof of \cite[Lemma~4.2]{FrMoPa}.
Moreover, it follows from \eqref{odde} that
\begin{eqnarray}\label{betatheta}
(d^j_1)^\prime\cdot d^j_2&=&(M^j)_{11}d_3^j\cdot d_2^j=0,\nonumber\\
(d^j_2)^\prime\cdot d^j_3&=&(M^j)_{12}d_3^j\cdot d^j_3=(M^j)_{12}\\
(d^j_1)^\prime\cdot d^j_3&=&(M^j)_{11}. \nonumber
\end{eqnarray}

Since the functions $p_j$ are smooth on the interval $\o I$, they can be extended smoothly to $\R$.
For $(t,s)\in \R^2$ we consider
$$
\Phi_j(t,s):=\big(t-s p_j(t){\,\cdot\,} \u e_2\big)\u e_1+s p_j(t){\,\cdot\,} \u e_1\,\u e_2.
$$
Since 
\begin{equation}\label{nfij}
\nabla \Phi_j(t,s)=
\left(
\begin{array}{cc}
1-s p_j'(t){\,\cdot\,} \u e_2 & -p_j(t){\,\cdot\,} \u e_2\smallskip\\
s p_j'(t){\,\cdot\,} \u e_1 & p_j(t){\,\cdot\,} \u e_1
\end{array}
\right),
\end{equation}
we have that
$$
(\det\nabla \Phi_j)(t,0)=p_j(t){\,\cdot\,} \u e_1>0\quad \mbox{ for every }t\in\o I,
$$
hence, there exist an interval $I_j \supset\o I$ and $\eta_j>0$ such that
$$
\det\nabla \Phi_j>0\quad \mbox{ on }\o I_j\times [-\eta_j,\eta_j].
$$
By the Implicit Function Theorem there exists $\rho>0$ such that, if $(t,s),(t',s')\in\o I_j\times [-\eta_j,\eta_j]$, then
$$
0<|t-t'|^2+|s-s'|^2\leq \rho^2 \quad \Rightarrow \quad  \Phi_j(t,s)\neq  \Phi_j(t',s').
$$
On the other hand, using the definition of $\Phi_j$, we have that, up to choosing $\eta_j$ smaller if needed,  
there exists a constant $c>0$ such that, if $(t,s),(t',s')\in\o I_j\times [-\eta_j,\eta_j]$, then
$$
|t-t'|\geq\frac\rho2  \quad \Rightarrow \quad  |\Phi_j(t,s)-\Phi_j(t',s')|\geq c.
$$
These two facts together imply that, up to choosing $\eta_j$ smaller, $\Phi_j$ is injective on $\o I_j\times [-\eta_j,\eta_j]$.
By the invariance of domain theorem the set $U_j:=\Phi_j(I_j\times (-\eta_j,\eta_j))$ is open and, since both $\nabla \Phi_j$ and $(\det \nabla \Phi_j)^{-1}$ are continuous on $\o I_j\times [-\eta_j,\eta_j]$, the inverse $\Phi_j^{-1}$ belongs to $C^1(\o U_j;\R^2)$.
Since 
$$
\Phi_j(t,0)=t\u e_1 \quad \text{ for every } t\in I_j,
$$
there exists $\e_j>0$ such that $S_\e\subset U_j$ for any $0<\e\le\e_j$, in other words $\Phi_j^{-1}$ is defined on $S_{\e}$ for  $0<\e\le\e_j$.

For $(t,s)\in I\times\R$ we now define
$$
v_j(t,s) := y^j(t)+s b_j(t)
$$
where 
$$
b_j(t) :=  -p_j(t){\,\cdot\,} \u e_2\, d_1^j(t)+p_j(t){\,\cdot\,} \u e_1\,
d_2^j(t)
$$
for every $t\in I$, and
$$
u_j(x) := v_j\big(\Phi_j^{-1}(x) \big)
$$
for $x\in S_\e$.

By the definitions above
\begin{equation}\label{nuj}
\nabla v_j=((y^j)^\prime +s b_j^\prime | b_j),
\qquad
(\nabla u_j)(\Phi_j)\nabla \Phi_j =\nabla v_j.
\end{equation}
By means of \eqref{betatheta} one can check that
$$
(b_j)^\prime\cdot d_3^j=0,\qquad |b_j^\prime|=|p_j^\prime|.
$$
With these identities at hand one can show that $(\nabla v_j)^T\nabla v_j=(\nabla \Phi_j)^T\nabla \Phi_j$, that is, $(\nabla u_j)^T\nabla u_j=I$.
Clearly, $u_j(\cdot, 0) = y^j$ and $\d_1 u_j(\cdot, 0) = (y^j)'=d^j_1$.
Moreover,
we have
\begin{equation}
\label{note-1}
(\D u_j)(\Phi_j)\d_s \Phi_j = \d_s v_j = b_j.
\end{equation}
Using that $p_j\cdot \u e_1\neq 0$, one readily deduces that
\begin{equation}
\label{gradient}
\D u_j(\cdot, 0) = ( d_1^j \, |\, d_2^j).
\end{equation}
Taking derivatives with respect to $s$ on both sides of \eqref{note-1},
we see that 
$$
\D^2 u_j(\Phi_j)\d_s \Phi_j\cdot \d_s \Phi_j = 0,
$$
and therefore
\begin{equation}
\label{note}
A_{u_j}(\Phi_j)\d_s \Phi_j\cdot \d_s \Phi_j = 0.
\end{equation}
Taking derivatives in \eqref{gradient}, we see
that
$$
(A_{u_j}(\cdot, 0))_{11} =
d_3^j\cdot \d_1\d_1 u_j(\cdot, 0) = d_3^j\cdot (d_1^j)' = (M^j)_{11}
$$
and similarly that $(A_{u^j}(\cdot, 0))_{12} = (M^j)_{12}$. 
On the other hand, since $\d_s \Phi_j$ is orthogonal to $p_j$, we have that $M^j\d_s \Phi_j=0$. Using \eqref{note} and the fact that $p_j\cdot \u e_1\neq 0$, we conclude that $A_{u^j}(\cdot, 0) = M^j$.

We now prove that $u_j$ satisfies the boundary conditions, namely $u_j\in\AA_\e$. To this aim, we remark that 
condition (i) ensures that $\Phi_j(0,s)=(0,s)$ and $\Phi_j(\ell,s)=(\ell,s)$ for every $s$.
This implies that $u_j(0,0)=v_j(0,0)=y^j(0)=0$ and $u_j(\ell,0)=v_j(\ell,0)=y^j(\ell)=\o y$.
Again by condition (i) we have that $p_j'(0)=p_j'(\ell)=0$ and $b_j=d_2^j$ close to $\d I$, hence
$b_j'(0)=(d_2^j)'(0)$ and $b_j'(\ell)=(d_2^j)'(\ell)$. From \eqref{odde} it follows that $r_j^\prime(0)=A_{M_j}(0)r_j(0)=A_{\lambda_j\u e_1\otimes \u e_1}$, hence $(d_2^j)^\prime(0)=0$. Analogously, one can show that $(d_2^j)^\prime(\ell)=0$.
By \eqref{nfij} we deduce that $\nabla \Phi_j(0,s)=\nabla \Phi_j(\ell,s)=I$ for every $s$.
We now use \eqref{nuj} to conclude that $\nabla u_j(0,x_2)=(d_1^j(0)\, |\, b_j(0))$
and $\nabla u_j(\ell,x_2)=(d_1^j(\ell)\, |\, b_j(\ell))$ for $x_2\in(-\e_j,\e_j)$. By the initial condition in \eqref{nfij} and the fact that $p_j(0)=\u e_1$
we have that $\nabla u_j(0,x_2)= (e_1\, |\, e_2)$ for $x_2\in(-\e_j,\e_j)$. Since $r_j(\ell)=\o r$ and $p_j(\ell)=\u e_1$, we have that $\nabla u_j(\ell,x_2)= (\o d_1|\o d_2)$ for $x_2\in(-\e_j,\e_j)$.

We are now in a position to define the recovery sequence.
For $\e$ small enough, the maps $y^j_{\e} : S\to\R^3$ given by
$y^j_{\e}(x_1, x_2) = u^j(x_1, \e x_2)$
are well-defined scaled $C^2$-isometries of $S$ such that
$$
\De y^j_{\e} = (\D u^j)(T_{\e}) \to  \D u^j(\cdot, 0) =
( d_1^j \, |\, d_2^j )
\qquad \mbox{ strongly in } W^{1,2}(S; \R^{3\times2}),
$$
as $\e \to 0$; here $T_{\e}x = (x_1, \e x_2)$.
Set $A_{\e}^j:= A_{y^j_\e,\e}$. Then since $A_{u^j}(x_1, 0) = M^j(x_1)$,
we see that $A_{\e}^j\to M^j$ strongly in $L^2(S;\R^{2\times2}_{\sym})$, as $\e\to 0$.
Hence,
$$
\lim_{\e\to 0} J_\e(y^j_{\e})=
\lim_{\e\to 0} \int_S |A_{\e}^j|^2\, dx
= \int_0^\ell |M^j|^2\, dt  .
$$
By a diagonal argument we obtain the desired result.
\end{proof}

\section{Equilibrium equations for the Sadowsky functional}\label{eqeq}

In this section we consider the minimization problem for the Sadowsky functional \eqref{defJ}
on the class $\AA_0$ introduced in \eqref{AA0} and the corresponding Euler-Lagrange equations.

\subsection{Existence of a solution}
Let  $\chi_{\AA_0}$ be the indicator function of the set $\AA_0$.
Since the density function $\o{Q}$ is convex and $\o{Q}(\mu,\tau)\ge \mu^2+\tau^2$ for any $\mu,\tau\in\R$,  
it is easy to prove that the functional 
$E+\chi_{\AA_0}$ is $(W^{2,2}\times W^{1,2})$-weakly lower semicontinuous and coercive. By the direct method there exists a minimizer of $E$ in $\AA_0$.
Uniqueness is not ensured since $\o{Q}$ is not strictly convex.

\subsection{Euler-Lagrange equations}  

Let $|\o y|<\ell$ and let $(y,r)$ be a minimizer of $E$ on $\AA_0$. We always write $r^T=(d_1|d_2|d_3)$.
A system of Euler-Lagrange equations for $(y,r)$ has been derived in \cite{H2010}. 
In that paper the energy is considered as a function of $\kappa=d_1'\cdot d_2$, $\mu=d_1'\cdot d_3$ and $\tau=-d_2'\cdot d_3$, so that  
the constraint $d_1^\prime\cdot d_2=0$ corresponds to assuming $\kappa=0$. Taking 
$\dot\kappa=0$ in \cite[eq.\,(12)]{H2010}, one obtains that $(y,r)$ satisfies the second and the third equation in \cite[eq.\,(14)]{H2010}.
Note that, since $|\o y|<\ell$, one can rule out the degenerate case where the curve $y$ is a straight line. From these considerations we obtain the following proposition.

\begin{proposition}[Equilibrium equations]\label{see}
Let $|\o y|<\ell$ and let $(y,r)$ be a minimizer of $E$ on $\AA_0$ with $r^T=(d_1|d_2|d_3)$. Then
there exist Lagrange multipliers $\lambda_1,\lambda_2\in\R^3$  such that the following boundary value problem is satisfied:
\begin{equation}\label{scbc}
\begin{cases}
\dfrac{\d\o Q}{\d\tau}(d_1'\cdot d_3, d_2'\cdot d_3)=\big(\lambda_2+\lambda_1\wedge y\big){\,\cdot\,} d_1 &
\text{ a.e.\ in } (0,\ell), \smallskip\\
\dfrac{\d\o Q}{\d\mu}(d_1'\cdot d_3, d_2'\cdot d_3)=\big(\lambda_2+\lambda_1\wedge y\big){\,\cdot\,} d_2 &
\text{ a.e.\ in } (0,\ell), \smallskip\\
(d_1,d_2,d_3)\in SO(3) &
\text{ a.e.\ in } (0,\ell), \smallskip\\
d_1'\cdot d_2=0\ &
\text{ a.e.\ in } (0,\ell), \smallskip\\
y(0)=0,\ y(\ell)=\o y \smallskip\\
r(0)=I,\ r(\ell)=\o r.
\end{cases}
\end{equation}
\end{proposition}

\begin{proof}
The second and the third equation in \cite[eq.\,(14)]{H2010} are
\begin{eqnarray*}
\frac{\d\o Q}{\d\tau}(d_1'(t)\cdot d_3(t), d_2'(t)\cdot d_3(t)) &= &\Big(\lambda_2-\lambda_1\wedge\int_t^\ell d_1(s)\,ds\Big)\cdot d_1(t),\\
\frac{\d\o Q}{\d\mu}(d_1'(t)\cdot d_3(t), d_2'(t)\cdot d_3(t)) & = & \Big(\lambda_2-\lambda_1\wedge\int_t^\ell d_1(s)\,ds\Big)\cdot d_2(t).
\end{eqnarray*}
Since $d_1=y'$, we have that $\int_t^\ell d_1(s)\,ds=\o y-y(t)$. The thesis follows
by replacing the multiplier $\lambda_2$ with $\lambda_2-\lambda_1\wedge \o y$.
\end{proof}

\begin{remark}
The two partial derivatives on the left-hand side of the first two equations in \eqref{scbc} have the mechanical meaning of a twisting and bending moment, respectively, and are given by
\begin{eqnarray*}
\frac{\d\o Q}{\d\tau}(\mu,\tau)&=&
\begin{cases}
\dis 4\tau\frac{{\mu}^2+{\tau}^2}{{\mu}^2}&\mbox{ if }|\mu|>|\tau|,\\
8\tau&\mbox{ if }|\mu|\le|\tau|,
\end{cases}\\
\frac{\d\o Q}{\d\mu}(\mu,\tau)&=&
\begin{cases}
\dis 2\frac{{\mu}^4-{\tau}^4}{{\mu}^3}&\mbox{ if }|\mu|>|\tau|,\\
0&\mbox{ if }|\mu|\le|\tau|.
\end{cases}
\end{eqnarray*}
\end{remark}

The next proposition shows that on a closed planar curve like a circle, where the curvature is always positive, 
the constraint $d_1'\cdot d_2=0$ is incompatible with the boundary condition $\o r=(e_1|-e_2|-e_3)$.
As a consequence, the centerline of a developable M\"obius band, if planar, must contain a segment. 

\begin{proposition}\label{pcd2c}
Let $(y, r)\in W^{2,2}(I; \R^3)\times W^{1,2}(I; SO(3))$ be such that with $r^T=(d_1|d_2|d_3)$, $y'=d_1$, and $d_1'\cdot d_2=0$ a.e.\ on $I$.
If the curve $y$ is planar and $d_1'\cdot d_3>0$ a.e.\ in $I$, then $d_2$ is constant and orthogonal to the plane of the curve.
\end{proposition}

\begin{proof} 
Let $\mu:=d_1'\cdot d_3$ and let $k$ be a normal unit vector to the plane of the curve, that is, $|k|=1$ and $d_1\cdot k\equiv 0$.
Since $d_1'=\mu d_3$, we have that $d_1\wedge d_1'=-\mu d_2$. Using the fact that $\mu(s)>0$ for a.e.\ $s\in I$, we deduce that $d_2=-d_1\wedge d_1'/\mu$ a.e.\ in $I$. 
On the other hand, both $d_1$ and $d_1'$ are orthogonal to $k$ (indeed,
$d_1'\cdot k=(d_1\cdot k)'=0$), hence $d_2$ is parallel to $k$ a.e.\ in $I$.  By continuity $d_2$ must be constantly equal to $k$ or to $-k$ in $I$.
\end{proof}

\section{Regular M\"obius bands at equilibrium}\label{lastsec}

It is easy to construct a developable M\"obius band by adding segments to the centerline in a way that it remains planar, see Example \ref{exnomin}.
On the other hand, we can show that a developable M\"obius band, whose centerline is regular and planar, cannot satisfy the equilibrium equations. In other words, the centerline of a regular developable M\"obius band at equilibrium cannot be planar.

The regularity notion that we need, is the following.

\begin{definition}\label{partition} A solution $(y, r)\in W^{2,2}(I; \R^3)\times W^{1,2}(I; SO(3))$ of the boundary value problem \eqref{see}--\eqref{scbc} with $r^T=(d_1|d_2|d_3)$
is said to be {\it regular} 
if  there exists a family $\mathscr{F}$ of pairwise disjoint open subintervals of $(0,\ell)$
such that $|(0,\ell)\setminus \cup\mathscr{F}|=0$ and such that
in every open interval $J\in\mathscr{F}$  the curvature $\mu=d_1'\cdot d_3$ is (a.e.) either strictly positive, strictly negative, or zero. 
\end{definition}

\begin{remark}
An example of a function $\mu\in L^2(0,1)$ that is not regular in the sense of Definition~\ref{partition} is
the characteristic function of a fat Cantor set in $(0,1)$ (or of any closed set with positive measure and empty interior).
\end{remark}

In the next theorem we show that a regular solution of the Euler-Lagrange equations with boundary conditions $\o y=0$ and  $\o r=(e_1|-e_2|-e_3)$ cannot be planar. In the proof we use in a crucial way the expression of $\o Q$ for small curvatures, which is exactly the region where $\o Q$ differs from the classical Sadowsky energy density.

\begin{theorem}\label{cnp}
Assume that $(y,r)$ be a regular solution of \eqref{scbc} with $\o y=0$ and $\o r=(e_1|-e_2|-e_3)$. 
Then the curve $y$ is not planar, i.e.,  $d_1$ does not lie on a plane. 
\end{theorem}

\begin{proof} 
Assume by contradiction that $d_1$ lies on a plane, which we assume to
be orthogonal to a unit vector $k$, that is,  $d_1(s)\cdot k=0$ for every
$s\in(0,\ell)$. We set $\mu:=d_1'\cdot d_3$ and $\tau:=-d_2'\cdot d_3$.

Let $\mathscr{F}$ be a family of pairwise disjoint open subintervals of $(0,\ell)$
as in Definition~\ref{partition} and let $J\in \mathscr{F}$.
If $\mu(s)>0$ for a.e.\ $s\in J$, then, by applying Proposition~\ref{pcd2c} on the interval $J$ we have that $d_2$ is constantly equal to $k$ or $-k$ in $J$, hence $d_2'=0$. By definition of $\tau$, 
this implies that $\tau=0$ in $J$. 
The same conclusion is true if $\mu(s)<0$ for a.e.\ $s\in J$. 
If, instead, $\mu(s)=0$ for a.e.\ $s\in J$, then the curvature vanishes on the interval and the curve is a segment. 
Therefore, globally the curve is a union of segments and of arcs with $\tau=0$.

Let us set $S:=\cup \{J\in\mathscr{F}:\ \mu=0\mbox{ a.e.\ in }J\}$ and $C:=\cup\{J\in\mathscr{F}:\ \mu\ne0\mbox{ a.e.\ in }J\}$. 

On the set $S$ we have $\mu=0$, which corresponds to the regime $|\mu|\le|\tau|$, while on $C$ we have $\tau=0$,  
corresponding to the opposite regime $|\mu|>|\tau|$.

Hence, using also the fact that $y(\ell)=0$, equations \eqref{scbc} become
\begin{equation}\label{eeonS}
\begin{cases}
8\tau(t)=\Big(\lambda_2+\lambda_1\wedge y(t)\Big)\cdot d_1(t)\\
0=\Big(\lambda_2+\lambda_1\wedge y(t)\Big)\cdot d_2(t)
\end{cases}
\mbox{ on }S,
\end{equation}
\begin{equation}\label{eeonC}
\begin{cases}
0=\big(\lambda_2+\lambda_1\wedge y(t)\big)\cdot d_1(t)\\
2\mu(t)=\big(\lambda_2+\lambda_1\wedge y(t)\big)\cdot d_2(t)
\end{cases}
\mbox{ on }C.
\end{equation}
Then we have
\begin{equation}\label{eeglob}
\lambda_2-\lambda_1\wedge y(t)=
\begin{cases}
8\tau(t)\, d_1(t)+\big(\lambda_2+\lambda_1\wedge y(t)\big)\cdot d_3(t)\,d_3(t)&\mbox{ on }S\\
2\mu(t)\, d_2(t)+\big(\lambda_2+\lambda_1\wedge y(t)\big)\cdot d_3(t)\,d_3(t)&\mbox{ on }C
\end{cases}
\end{equation}
and, as a consequence, the curvatures $\mu$ and $\tau$ turn out to be continuous on the sets $S$ and $C$.
By the closure assumption $\o y=0$, we have that $C\ne\varnothing$. 
On the other hand, by the twisting assumption $d_2(\ell)=-e_2=-d_2(0)$ we have that $S\ne\varnothing$, too.
By the first equation in \eqref{eeonS} and the second equation in \eqref{eeonC} we also have that 
at the boundary points between the two regions $S$ and $C$ there exist finite the limits from the left and from the right
of $\mu$ and $\tau$. Denoting by $t_0$ one of such points, with $S$ on the left and $C$ on the right to fix ideas,
by continuity of the left-hand side in \eqref{eeglob} we have that
$$
8\tau(t_0^-)d_1(t_0)=2\mu(t_0^+)d_2(t_0),
$$
where we have set
$$
\tau(t_0^-):=\lim_{t\to t_0^-}\tau(t), \qquad \mu(t_0^+):=\lim_{t\to t_0^+}\mu(t).
$$
Since $d_1(t_0)$ and $d_2(t_0)$ are orthogonal, we deduce that $\tau(t_0^-)=\mu(t_0^+)=0$. 
On the other hand, $\tau=0$ on $C$ and $\mu=0$ on $S$, hence $\tau(t_0^+)=\mu(t_0^-)=0$. Thus, we conclude
that the functions $\mu$ and $\tau$ are globally continuous and must vanish at the boundary points between $S$ and $C$.

From the first equation in \eqref{eeonS} it follows that $\tau$ is constant.
Indeed, on every interval $J$ of  $S$ the curve $y$ is affine, that is, $y$ is of the form $y(t)=d_1t+c$ with constants $d_1$ and $c$. Therefore,
on $J$ we have
\begin{eqnarray*}
8\tau(t)
&=&
\Big(\lambda_2+\lambda_1\wedge y(t)\Big)\cdot d_1\\
&=&\lambda_2\cdot d_1+t\lambda_1\wedge d_1 \cdot d_1 +\lambda_1\wedge c\cdot d_1\\
&=&
\lambda_2\cdot d_1+\lambda_1\wedge c\cdot d_1.
\end{eqnarray*}
Being constant and equal to zero at the boundary points, $\tau=0$ on $S$. 
We conclude that $\tau=0$ on $[0,\ell]$ and this gives a contradiction, since the boundary condition on $r$ cannot be satisfied.
\end{proof} 

We conclude with an explicit example of a developable M\"obius band, whose centerline is planar. Because of the previous theorem 
it cannot satisfy the equilibrium equations and thus, it cannot be a minimizer of the Sadowsky functional \eqref{defJ}.

\begin{example}[Non-minimal developable M\"obius band]\label{exnomin} 
We consider the framed curve $(y,r)$ given by
$$
y(t)=\begin{cases}
(t,0,0)&\mbox{ if }t\in(0,1),\\
(1+\sin(t-1),0,1-\cos(t-1))&\mbox{ if }t\in(1,1+\pi),\\
(2+\pi-t,0,2)&\mbox{ if }t\in(1+\pi,2+\pi),\\
(\sin(t-2),0,1-\cos(t-2))&\mbox{ if }t\in(2+\pi,2+2\pi),\\
\end{cases}
$$
$$
d_1(t)=\begin{cases}
(1,0,0)&\mbox{ if }t\in(0,1),\\
(\cos(t-1),0,\sin(t-1))&\mbox{ if }t\in(1,1+\pi),\\
(-1,0,0)&\mbox{ if }t\in(1+\pi,2+\pi),\\
(\cos(t-2),0,\sin(t-2))&\mbox{ if }t\in(2+\pi,2+2\pi),\\
\end{cases}
$$
$$
d_2(t)=\begin{cases}
(0,\cos(\pi t),\sin(\pi t))&\mbox{ if }t\in(0,1),\\
(0,-1,0)&\mbox{ if }t\in(1,1+\pi),\\
(0,-1,0)&\mbox{ if }t\in(1+\pi,2+\pi),\\
(0,-1,0)&\mbox{ if }t\in(2+\pi,2+2\pi),\\
\end{cases}
$$
$$
d_3(t)=d_1(t)\wedge d_2(t)=\begin{cases}
(0,-\sin(\pi t),\cos(\pi t))&\mbox{ if }t\in(0,1),\\
(\sin(t-1),0,-\cos(t-1))&\mbox{ if }t\in(1,1+\pi),\\
(0,0,1)&\mbox{ if }t\in(1+\pi,2+\pi),\\
(\sin(t-2),0,-\cos(t-2))&\mbox{ if }t\in(2+\pi,2+2\pi).\\
\end{cases}
$$
The boundary conditions $y(0)=y(\ell)=0$, $r(0)=I$ and  $r(\ell)=(e_1|-e_2|-e_3)$ with $\ell=2+2\pi$ are satified.
For $t\in(0,1)$ the director $d_2$ rotates from $e_2$ to $-e_2$, while it is constantly equal to $-e_2$ for $t\in(1,\ell)$. Since $d_1$ is constant on $(0,1)$ and it always belongs to the plane $x_1x_3$, we have that $d_1'\cdot d_2=0$, hence $(y,r)\in\AA_0$.
However, $(y,r)$ cannot be a minimizer since the curve $y$ belongs to the plane $x_1x_3$.
\end{example}

\bigskip\bigskip

\noindent
{\bf Acknowledgements.} The authors would like to thank Gianni Dal Maso for several discussions about the content of Section~\ref{lastsec}. The work of LF has been supported by DMIF--PRID project PRIDEN. MGM acknowledges support by MIUR--PRIN 2017. LF and MGM are members of GNAMPA--INdAM and RP is a member of GNFM--INdAM.


\begin{thebibliography}{99}

\bibitem{AlexanderAntman}
J.C.~Alexander, S.S.~Antman: The ambiguous twist of Love, 
{\it Quart. Appl. Math.} {\bf 40} (1982), 83--92.

\bibitem{ADK}
V.~Agostiniani, A.~De Simone, K.~Koumatos: Shape programming for narrow ribbons of nematic elastomers, 
{\it J. Elasticity} {\bf 127} (2017), 1--24. 

\bibitem{AN}
B.~Audoly, S.~Neukirch: A one-dimensional model for elastic ribbons: a little stretching makes a big difference, 
{\it J. Mech. Phys. Solids} {\bf 153} (2021), 104--157. 

\bibitem{AS}
B.~Audoly, K.A.~Seffen: Buckling of naturally curved elastic strips: the ribbon model makes a difference, 
{\it J. Elasticity} {\bf 119} (2015), 293--320.

\bibitem{Ba}
S.~Bartels: Numerical simulation of inextensible elastic ribbons,
{\it SIAM J. Numer. Anal.} {\bf 58} (2020), 3332--3354.

\bibitem{BH2015}
S. Bartels, P. Hornung: Bending paper and the M\"{o}bius strip, 
{\it J. Elasticity} {\bf 119} (2015), 113--136.

\bibitem{BFV}
M.~Brunetti, A.~Favata, S.~Vidoli:
Enhanced models for the nonlinear bending of planar rods: localization phenomena and multistability,
{\it Proc. Roy. Soc. Edinburgh Sect. A} {\bf 476} (2020), 20200455 (20 pp).

\bibitem{CDNR} 
R.~Charrondi\`ere, F.~Bertails-Descoubes, S.~Neukirch, V.~Romero: Numerical modeling of inextensible elastic ribbons with curvature-based elements, 
{\it Comput. Methods Appl. Mech. Engrg.} {\bf 364} (2020), 112922 (24 pp).

\bibitem{CDD}
J.~Chopin, V.~D\'emery, B.~Davidovitch: Roadmap to the morphological instabilities of a stretched twisted ribbon, 
{\it J. Elasticity} {\bf 119} (2015), 137--189.

\bibitem{Da}
E.~Davoli: Thin-walled beams with a cross-section of arbitrary geometry: derivation of linear theories starting from 3D nonlinear elasticity,
{\it Adv. Calc. Var.} {\bf 6} (2013), 33--91.

\bibitem{DA2015}
M.A.~Dias, B.~Audoly: ``Wunderlich, Meet Kirchhoff": A general and unified description of elastic ribbons and thin rods, 
{\it J. Elasticity} {\bf 119} (2015), 49--66.

\bibitem{EberhardHornung}
P.~Eberhard, P.~Hornung: On singularities of stationary isometric deformations,  
{\it Nonlinearity} {\bf 33} (2020), 4900--4923.


\bibitem{FoFr}
R.\ Fosdick, E.\ Fried, editors: 
{\it The mechanics of ribbons and M\"obius bands},
Springer, 2015.

\bibitem{FrHoMoPa}
L.~Freddi, P.~Hornung, M.G.~Mora, R.~Paroni:
A corrected Sadowsky functional for inextensible elastic ribbons,
{\it J. Elasticity} {\bf 123} (2016), 125--136.

\bibitem{longversion}
L.~Freddi, P.~Hornung, M.G.~Mora, R.~Paroni:
A variational model for anisotropic and naturally twisted ribbons,
{\it  SIAM J. Math. Anal.} {\bf 48} (2016), 3883--3906.

\bibitem{FrHoMoPa2}
L.~Freddi, P.~Hornung, M.G.~Mora, R.~Paroni:
One-dimensional von K\'arm\'an models for elastic ribbons,
{\it Meccanica} {\bf  53} (2018), 659--670.

\bibitem{FrMoPa}
L.~Freddi, M.G.~Mora, R.~Paroni:
Nonlinear thin-walled beams with a rectangular cross-section -- Part I,
{\it Math. Models Methods Appl. Sci.} {\bf 22} (2012), 1150016 (34 pp).

\bibitem{FrMoPa2}
L.~Freddi, M.G.~Mora, R.~Paroni:
Nonlinear thin-walled beams with a rectangular cross-section -- Part II,
{\it Math. Models Methods Appl. Sci.} {\bf 23} (2013), 743--775.

\bibitem{FM} 
M.~Friedrich, L.~Machill:
Derivation of a one-dimensional von K\'arm\'an theory for viscoelastic ribbons, 
Preprint arXiv 2021, arXiv:2108.05132.

\bibitem{HF2015b}
D.F.~Hinz, E.~Fried:
Translation of Michael Sadowsky's Paper ``An elementary proof for the existence of a developable M\"{o}bius band and the attribution of the geometric problem to a variational problem", 
{\it J. Elasticity} {\bf 119} (2015), 3--6.

\bibitem{HF2015a}
D.F.~Hinz, E.~Fried:
Translation and interpretation of Michael Sadowsky's paper ``Theory of elastically bendable inextensible bands with applications to the M\"{o}bius band", 
{\it J. Elasticity} {\bf 119} (2015), 7--17.

\bibitem{H2010}
P.~Hornung: 
Euler-Lagrange equations for variational problems on space curves,
{\it Phys. Rev. E} {\bf 81} (2010), 066603 (5 pp).

\bibitem{framed}
P.~Hornung:
Deformation of framed curves with boundary conditions,
{\it Calc. Var. Partial Differential Equations} {\bf 60} (2021), 87 (26 pp).

\bibitem{H2}
P.~Hornung:
Deformation of framed curves,
Preprint arXiv 2021, 	arXiv:2110.08541.
              
\bibitem{KAB} 
K.~Korner, B.~Audoly, K.~Bhattacharya: 
Simple deformation measures for discrete elastic rods and ribbons, 
Preprint arXiv 2021, arXiv:2107.04842.

\bibitem{KuHB} 
A.~Kumar, P.~Handral, C.S.D.~Bhandari, A.~Karmakar, R.~Rangarajan: 
An investigation of models for elastic ribbons: Simulations \& experiments, 
{\it J. Mech. Phys. Solids} {\bf143} (2020), 104070 (37 pp).

\bibitem{LSSM}
I.~Levin, E.~Si\'efert, E.~Sharon, C.~Maor:
Hierarchy of geometrical frustration in elastic ribbons: Shape-transitions and energy scaling obtained from a general asymptotic theory,
{\it J. Mech. Phys. Solids} {\bf 156} (2021), 104579 (14 pp).

\bibitem{MH18} 
A.~Moore, T.~Healey: 
Computation of elastic equilibria of complete M\"obius bands and their stability,
{\it Math. Mech. Solids} {\bf 24} (2018), 939--967.

\bibitem{PT}
R.~Paroni, G.~Tomassetti: 
Macroscopic and microscopic behavior of narrow elastic ribbons, 
{\it J. Elasticity} {\bf 135} (2019), 409--433.

\bibitem{Sadowsky}
M.~Sadowsky:
Ein elementarer Beweis f\"ur die Existenz eines abwickelbaren M\"obiusschen Bandes und
die Zur\"uckf\"uhrung des geometrischen Problems auf ein Variationsproblem,
Sitzungsber. Preuss. Akad. Wiss. (1930), Mitteilung vom 26 Juni, pp. 412--415.

\bibitem{Sadowsky2}
M.~Sadowsky:
Theorie der elastisch biegsamen undehnbaren B\"ander mit Anwendungen auf das M\"obiussche Band,
{\it Verhandl. des 3. Intern. Kongr. f. Techn. Mechanik} {\bf 2} (1930), 444--451.

\bibitem{SH2007}
E.L.~Starostin, G.H.M.~van der Heijden: 
The equilibrium shape of an elastic developable M\"{o}bius strip, 
{\it PAMM Proc. Appl. Math. Mech.} {\bf 7} (2007), 2020115--2020116.

\bibitem{SH2015}
E.L.~Starostin, G.H.M.~van der Heijden:
Equilibrium shapes with stress localisation for inextensible elastic M\"{o}bius and other strips, 
{\it J. Elasticity} {\bf 119} (2015), 67--112.

\bibitem{TeV}
L.~Teresi, V.~Varano: 
Modeling helicoid to spiral-ribbon transitions of twist-nematic elastomers, 
{\it Soft Matter} {\bf 9} (2013), 3081--3088.

\bibitem{ToV}
G.~Tomassetti, V.~Varano: Capturing the helical to spiral transitions in thin ribbons of nematic elastomers,
{\it Meccanica} {\bf 52} (2017), 3431--3441.

\bibitem{Yu}
T.~Yu:
Bistability and equilibria of creased annular sheets and strips, 
Preprint arXiv 2021, arXiv:2104.09704.

\bibitem{YDMG}
T.~Yu, L.~Dreier, F.~Marmo, S.~Gabriele, S.~Parascho, S.~Adriaenssens:
Numerical modeling of static equilibria and bifurcations in bigons and bigon rings, 
{\it J. Mech. Phys. Solids} {\bf 152} (2021), 104459 (28 pp).

\bibitem{YH}
T.~Yu, J.A.~Hanna: Bifurcations of buckled, clamped anisotropic rods and thin bands under lateral end translations, 
{\it J. Mech. Phys. Solids} {\bf 122} (2019), 657--685.



\end{thebibliography}
\end{document}